\documentclass{elsarticle}
\usepackage{amsmath}
\usepackage{amssymb}
\usepackage{verbatim}
\usepackage{graphicx}
\usepackage{color}
\usepackage{amssymb,bbm}

\newtheorem{theorem}{Theorem}[section]
\newtheorem{corollary}[theorem]{Corollary}
\newtheorem{definition}[theorem]{Definition}

\newtheorem{lemma}[theorem]{Lemma}
\newtheorem{remark}{Remark}[section]

\newcommand{\ms}{\mathbb{S}}
\newcommand{\mr}{\mathbb{R}}
\newcommand{\mcr}{\mathcal{R}}

\newcommand{\mF}{\mathcal{F}}
\newcommand{\mcs}{\mathcal{S}}
\newcommand{\mcp}{\mathcal{P}}
\newcommand{\mcd}{\mathcal{D}}
\newcommand{\mcc}{\mathcal{C}}
\newcommand{\mct}{\mathcal{T}}

\newcommand{\ext}{{\rm ext}}
\newcommand{\cor}{{\rm corank \hspace*{0.5mm}}}
\newcommand{\ran}{{\rm Ran \hspace*{0.5mm}}}
\newcommand{\Ker}{{\rm Ker \hspace*{0.5mm}}}
\newcommand{\rankspace}{{\rm rank \hspace*{0.5mm}}}
\newcommand{\gd}{{\rm gd}}
\newcommand{\p}{{\bf p}}

\newcommand{\q}{{\bf q}}
\newcommand{\E}{\overline{E}}

\newcommand{\per}{{\rm Pert}}

\newcommand{\gram}{{\rm Gram}}
\newcommand{\mcpa}{\mathcal{P}_G'}

\newcommand{\lin}{{\rm lin}}

\newcommand{\psd}{\succeq}

\newcommand{\la}{\langle}
\newcommand{\ra}{\rangle}

\newcommand{\sfT}{{\sf T}}

\newcommand{\oR}{{\mathbb R}}
\newcommand{\oN}{{\mathbb N}}

\newcommand{\corank}{\text{\rm corank}\hspace*{0.5mm}}

\begin{document}
\title{Positive Semidefinite Matrix Completion,  Universal Rigidity and the Strong Arnold Property}

\author[cwi,til]{M.~Laurent}
\ead{M.Laurent@cwi.nl}

\author[cwi]{A.~Varvitsiotis\corref{cor1}}
\ead{A.Varvitsiotis@cwi.nl}

\cortext[cor1]{Corresponding author: CWI, Postbus 94079,
	      1090 GB Amsterdam. Tel: +31 20 5924170; Fax: +31 20 5924199.}

\address[cwi]{Centrum Wiskunde \& Informatica (CWI), Science Park 123,
	        1098 XG Amsterdam,
	        The Netherlands.}
\address[til]{Tilburg University, 
P.O. Box 90153, 
5000 LE Tilburg, 
The Netherlands.}

\begin{abstract}
%This paper addresses three topics that at first sight might seem to be unrelated:  
This paper addresses the following three topics: positive semidefinite (psd) matrix completions,
  universal rigidity of frameworks, and    the Strong Arnold Property (SAP). 
%  We show some strong connections among these topics and the unifying theme of our treatment will be  semidefinite programming. 
  We show some strong connections among these topics, using semidefinite programming as unifying theme.
    Our main contribution  is a sufficient condition for constructing partial psd matrices which admit a unique completion to a full psd matrix. Such partial matrices are an essential  tool in the study  of  the Gram dimension $\gd(G)$ of a graph $G$, a recently studied graph parameter related to the low psd matrix completion problem. Additionally, we derive an elementary proof of Connelly's sufficient condition for universal rigidity of tensegrity frameworks and we investigate the links between these two sufficient conditions. We also  give a geometric characterization of psd matrices satisfying  the Strong Arnold Property in terms of nondegeneracy of an associated semidefinite program, which we use to establish some links between the Gram dimension   $\gd(\cdot)$ and the Colin de Verdi\`ere type graph parameter $\nu^=(\cdot)$. %by means of the extreme points of a certain spectrahedron. Lastly, we investigate the links between the two graph parameters $\gd(\cdot)$ and $\nu^=(G)$.
\end{abstract}

\maketitle

\section{Introduction}

The main motivation for this paper is  the  positive semidefinite (psd) matrix completion problem, defined as follows: Given a graph $G=(V=[n],E)$ and a  vector $a\in \mathbb{R}^{E\cup V}$ indexed by the nodes and the edges of $G$, decide whether there exists a   real symmetric $n\times n$ matrix $X$ satisfying
\begin{equation}\label{eq:psdcompl}
 X_{ij}=a_{ij}\  \text{ for all } \{i,j\} \in V\cup E, \text{ and } X  \text{ is positive semidefinite.}
 \end{equation}
% where $X \psd 0$ means that the matrix is positive semidefinite. 
 Throughout the paper  we identify $V$ with the set of diagonal pairs $\{i,i\}$ for $i\in V$.
 Any vector $a\in \oR^{V\cup E}$ can be viewed as a partial symmetric matrix  whose entries are determined only at the diagonal positions (corresponding to the nodes) and at the off-diagonal positions corresponding to the edges of   $G$. A vector  $a\in \oR^{V\cup E}$ is called a {\em $G$-partial psd matrix} when  (\ref{eq:psdcompl}) is feasible, i.e., when the partial matrix $a$  admits at least one completion to a full psd matrix. %For a graph $G$ we denote by $\ps$ the set of all $G$-partial psd matrices. 
 
  The psd matrix completion  problem   is  an instance of the  semidefinite programming feasibility problem, and as such its  complexity is still unknown~\cite{R97}. 
A successful line of attack embraced in the literature has been  to identify  graph classes for which  some of the handful of known necessary conditions that guarantee that a $G$-partial matrix is completable  are also  sufficient (see e.g. \cite{BJT,GJSW,Lau00}).

\medskip
In this paper we  develop a systematic method for  constructing    partial psd  matrices   with    the property that  they admit  a {\em unique} completion to a fully specified psd matrix. Such partial matrices   are a crucial ingredient for the study of two new  graph parameters considered in \cite{ELV, LV,LV12}, defined in terms of ranks of psd matrix completions of $G$-partial matrices. The first one  is the Gram dimension $\gd(\cdot)$ which we will introduce in Section~\ref{sec:param} and whose study is motivated by the low rank psd matrix completion problem~\cite{LV,LV12}. The second  one is the extreme Gram dimension $\text{egd}(\cdot )$ whose study is motivated by its relevance to the bounded rank Grothendieck constant of a graph~\cite{ELV}.  
Several instances of partial  matrices with a unique psd completion were constructed in \cite{ELV,LV,LV12}, but the proofs were mainly by direct case checking. In this paper we give a sufficient condition for constructing partial psd matrices with a unique psd completion (Theorem \ref{thm:sphcorp}) and using this condition we can recover most examples  of \cite{ELV,LV,LV12} (see Section \ref{sec:examples}).
%Our motivation comes from the study of two  new graph parameters introduced in~\cite{LV} and~\cite{ELV} where a number of such instances were constructed.   
%Nevertheless, the proofs consisted mainly of case checking and do not appear to be amenable to generalization.  
%The present paper can be viewed as   an attempt to systematize the study initiated in~\cite{LV,ELV}.
%This will come as a
 %direct application of   Theorem \ref{main}, which gives a sufficient condition for   determining the unicity of a primal optimal solution to a semidefinite program.  
 % to the existence of a nondegenerate optimal dual solution.
% which can be used to recover most of the previously known results. 

The condition for  uniqueness  of a  psd completion suggests a  connection to the theory of universally rigid frameworks. A {\em framework}  $G({\bf p})$  consists of  a graph  $G=(V=[n],E)$ together with an  assignment of vectors ${\bf p}=\{p_1,\ldots,p_n\}$ to the nodes of the graph. The framework $G(\p)$ is said to be  {\em universally rigid} if it is the only framework having the same edge  lengths  in any space, up to congruence.  
A related concept is that of global rigidity of frameworks. A framework $G(\p)$ in $\oR^d$  is called {\em globally rigid in $\oR^d$} if up to 
congruence it is the only framework in $\oR^d$ having the same edge  lengths. 
Both  concepts have been extensively studied and there exists  an abundant  literature about them 
(see e.g.~\cite{C82,Cb,Con,CW,GT} and references therein).    

The analogue of the notion  of global rigidity, in the case when  Euclidean distances are replaced by inner products, was   recently investigated in~\cite{SC}. There it is shown that many of the results that are valid in  the setting of Euclidean distances can be adapted  to the so-called `spherical setting'.  The latter terminology refers to the fact that when the vectors $p_1,\ldots,p_n\in \oR^d$ are restricted to lie on the unit sphere then their pairwise inner products lead to the study of the   spherical metric space, where the distance between two points $p_i,p_j$ is given by $\arccos (p_i^\sfT p_j)$, i.e., the angle formed between the two vectors~\cite{SW}.
Taking this analogy further, our sufficient condition for constructing partial psd matrices with a unique psd completion can be interpreted  as the analogue in the spherical setting of Connelly's celebrated sufficient condition for universal rigidity of frameworks (see the respective results from 
Theorem \ref{thm:sphcorp} and Theorem~\ref{thm:con}).

\medskip
The unifying  theme of this paper is semidefinite programming (SDP). In particular, the notions of  SDP nondegeneracy and  strict complementarity play   a crucial  role in this paper. This should come as no surprise   as there are already well established links  between  semidefinite programming  and universal rigidity \cite{A10b} and psd matrix completion with  SDP nondegeneracy \cite{Qi}. To arrive at our results we develop a number of tools that build  upon  fundamental  results guaranteeing   the uniqueness  of optimal  solutions to SDP's.  

Using this machinery we can also  give new proofs of some known results, most notably a short and elementary proof of  Connelly's sufficient condition for universal rigidity (Theorem \ref{thm:con}). With the intention to make Section~\ref{sec:con}  a self contained treatment  of universal rigidity we also  address the case of generic universally rigid frameworks (Section~\ref{sec:generic}).  Lastly,  we investigate the relation between our sufficient condition and Connelly's sufficient condition and show that in some special cases they turn out to be equivalent
(Section \ref{sec:connections}). 

\medskip
In this paper we  also  revisit  a  somewhat elusive matrix property called the {\em Strong Arnold Property} (SAP) whose study is motivated by the celebrated  Colin de Verdi\`ere graph parameter $\mu(\cdot)$ introduced in \cite{CdV98}.  We present  a geometric characterization  of matrices fulfilling the SAP by associating them with  the extreme points of a certain spectrahedron (Theorem \ref{thm:godsil}).  Furthermore, we show  that psd matrices having  the SAP can be understood as nondegenerate solutions of certain SDP's (Theorem \ref{thm:SAPtosdp}).

Lastly, using our tools  we can shed some more light and gain insight on the relation between two graph parameters that have been recently studied in the literature. 
 The first one is the  parameter $\nu^=(\cdot)$ of \cite{H96,H03}, whose study is motivated by its relation to  the Colin de Verdi\'ere graph parameter $\mu(\cdot)$.
The second one is  the {\em Gram dimension} $\gd(\cdot)$ of a graph,  introduced in \cite{LV,LV12},  whose study is motivated by   its relation  to the low rank psd matrix completion problem.
In particular we reformulate $\nu^=(\cdot)$ in terms of the maximum Gram dimension of certain $G$-partial psd matrices satisfying a  nondegeneracy property
(Theorem \ref{thmnewnu}), which enables us to recover  that $\gd(G)\ge \nu^=(G)$ for any graph $G$ (Corollary \ref{thm:gdnu}).

\subsection*{Contents.}
 The paper is organized as follows. In Section \ref{secSDP} we group  some basic facts about semidefinite programming that we need in the paper.
In Section \ref{secGram} we present our sufficient condition for the existence of unique psd completions  (in the general setting of tensegrities, i.e., allowing equalities and inequalities), and we illustrate its use by  several examples. In Section \ref{sec:con} we present a simple proof for  Connelly's sufficient condition for universally rigid  tensegrities (generic and non-generic) and we investigate the links between these two sufficient conditions for the spherical and Euclidean distance settings. Finally in Section \ref{sec:parameters} we revisit the Strong Arnold Property, we present a geometric characterization of psd matrices having the SAP in terms of nondegeneracy of semidefinite programming, which we use to establish a link between the graph parameters $\gd(\cdot)$ and $\nu^=(\cdot)$.

\subsection*{Notation.}
%We recall  some  basic concepts   from convex geometry,  a standard   reference  is~\cite{webster}.
 Let $C$ be a closed convex set. % in a finite dimensional vector space $ \mathcal{V}$. 
%The set $C$ is called {\em convex} if $\lambda C +(1-\lambda)C\subseteq C $ for any $\lambda \in [0,1]$. 
A convex subset $F\subseteq C$ is called a {\em face} of  $C$ if,  for any $x,y \in C$, $\lambda x+(1-\lambda)y\in F$ for some scalar $\lambda \in (0,1)$ implies $x,y\in F$.
% For $x \in C$,  $\mF_C(x)$ denotes the smallest face of $C$ containing  $x$. 
A  point $x\in C$ is called an  {\em extreme point} of $C$  if the set $\{x\}$ is  a face of $C$.
%the smallest face of $C$ containing $x$.  %We denote  $\ext (C)$ denotes the set of extreme points of  $C$. 
A vector $z$ is said to be  a {\em perturbation}  of $x \in C$ if  $x \pm \epsilon z \in C$ for some $\epsilon>0$. The set of perturbations of  $x \in C$ form  a linear space which we denote  as $\per_C(x)$. Clearly, $x$ is an extreme point of $C$ if and only if $\per_C(x)=\{0\}$. %Equivalently, one can easily verify that $\per_C(x)={\rm lin}\mF_C(x)$~\cite{P11}.  For a vector $x \in C$  it is immediate  that $x\in \ext C$ if and only if  $\per_C(x)=\{0\}$.

We denote by  $e_1,\cdots,e_n\in \oR^n$   the standard unit vectors in $\oR^n$  and for $1\le i\le j\le n$, we define the symmetric matrices
$E_{ij}=(e_ie_j^\sfT +e_je_i^\sfT)/2$ and $F_{ij}=(e_i-e_j)(e_i-e_j)^\sfT$. Given vectors $p_1,\ldots,p_n\in \oR^d$, $\lin\{p_1,\ldots,p_n\}$ denotes their linear span which is a vector subspace of $\oR^d$. We also use the shorthand notation $[n]=\{1,\ldots, n\}$.

 Throughout $\mcs^n$ denotes the set of real symmetric $n\times n $ matrices and $\mcs^n_+$ the subcone of positive semidefinite matrices.   For a matrix $X \in \mcs^n$ its kernel is  denoted as $\Ker X$ and its range as $\ran X$.
 The {\em corank} of a matrix $X\in \mcs^n$ is the dimension of its kernel.  For a matrix $X\in \mcs^n$, the notation $X\succeq 0$ means that $X$ is positive semidefinite (abbreviated as psd).
The space $\mcs^n$ is equipped with the trace inner product given by  $\langle X,Y\rangle =\text{Tr}(XY)=\sum_{i,j=1}^n X_{ij}Y_{ij}$.
%so  $\rankspace X +\corank X=n$.
We will use the following property: For two positive semidefinite matrices $X,Y\in\mcs^n_+$, $\la X,Y\ra \ge 0$, and $\la X,Y\ra =0$ if and only if $XY=0$.

Given vectors $p_1,\cdots,p_n\in \oR^d$,  their {\em Gram matrix} is the $n\times n$ symmetric matrix   $\gram(p_1,\ldots,p_n)=(p_i^\sfT p_j)_{i,j=1}^n$.
Clearly, the rank of the Gram matrix $\gram(p_1,\ldots,p_n)$ is equal to the dimension of the linear span of $\{p_1,\ldots,p_n\}$.
Moreover, two systems of vectors $\{p_1,\ldots,p_n\}$ and $\{q_1,\ldots,q_n\}$ in $\oR^d$ have the same Gram matrix, i.e., $p_i^\sfT p_j=q_i^\sfT q_j$ for all $i,j\in [n]$, if and only if 
%they are congruent\footnote{congruent is used mostly for EDM}, i.e., 
there exists a $d\times d$ orthogonal matrix $O$ such that $q_i=Op_i$ for all $i\in [n]$.

\section{Semidefinite programming}\label{secSDP}

%Our main goal in this section is to describe  a geometric criterion that characterizes instances of semidefinite programs that admit a unique optimal solution. This criterion constitutes our main result and it will be uses in subsequent sections to derive Connelly's sufficient condition for universal rigidity.
%In recent years semidefinite programming has become an essential tool  for the design of  efficient approximation algorithms for hard combinatorial optimization problems.
In this section we  recall some  basic facts  about semidefinite programming.  Our  notation and  exposition   follow  \cite{AHO} (another excellent source is \cite{Pa00}).

A semidefinite program is a convex program  defined as  the minimization of a linear function over an affine section of the cone of positive semidefinite matrices. In this paper we will consider  semidefinite programs of the form:
\begin{equation}\tag{P}\label{eq:primalsdp}
p^*= \underset{X}{\text{sup}} \left\{ \la C,X\ra  :    X \psd 0,\   \la A_i,X \ra=b_i \  (i\in I) , \ \la A_i,X \ra\le b_i  \ (i\in J)\right\}.
 \end{equation}
 Standard semidefinite programs are usually defined involving only linear equalities; we also allow here linear inequalities, since they will be used to model tensegrity frameworks in Sections \ref{secGram} and \ref{sec:con}. %with equality and inequality linear constraints.
The dual program of (\ref{eq:primalsdp}) reads:
\begin{equation}\tag{D}\label{eq:dualsdp}
d^*=\underset{y,Z}{\text{inf}}\left\{ \sum_{i\in I\cup J} b_iy_i\  :\  \sum_{i\in I\cup J}y_iA_i-C=Z \succeq 0,\  y_i \ge 0\  (i\in J) \right\}.
\end{equation}
Here, $C \in \mcs^n, A_i \in \mcs^n \ (i\in I\cup J)$ and $b\in \oR^{|I|+|J|}$ are given and $I\cap J=\emptyset$.

We denote the primal and dual feasible regions  by $\mcp$ and $\mcd$, respectively.  
  The primal feasible region 
\begin{equation}\label{specP}
\mcp=\left\{X\in \mcs^n:    X \psd 0,\   \la A_i,X \ra=b_i \  (i\in I) , \ \la A_i,X \ra\le b_i  \ (i\in J)\right\} 
 \end{equation}
 is a convex set  defined as  the intersection of the cone of positive semidefinite matrices with an affine subspace and some affine half-spaces. 
 For $J=\emptyset$, such sets are known as {\em spectrahedra} and, for $J\not=\emptyset$, they are called {\em semidefinite representable} (i.e., they can be obtained as projections of spectrahedra, by using slack variables).  Recently, there has been a surge of interest in the study of semidefinite  representable sets  since  they constitute a rich class of convex sets for which there exist  efficient algorithms for optimizing linear functions over them~\cite{BPT}.

As is well known (and easy to see), 
 weak duality holds: $p^*\le d^*$. Moreover, if the dual (resp. primal) is strictly feasible and $d^{*}>-\infty$ (resp. $p^*<\infty)$,
  then {\em strong duality} holds: $p^*=d^*$ and the primal (resp. dual) optimum value is attained.  

 A  pair  $X,(y,Z)$ of primal and dual optimal solutions are called {\em complementary}   if $XZ=0$ and    {\em strict complementary} if  moreover $\rankspace X+\rankspace Z=n$. For a matrix $X \in \mcp$, let $J_X=\{ i\in J : \la A_i,X\ra=b_i\}$ denote the set of inequality constraints that are  active at $X$.
Similarly, for a matrix $Z \in \mcd$ set $J_Z=\{i\in J : y_i>0\}$. Assuming strong duality,  a pair of primal and dual  feasible solutions $X, (y,Z)$ are both optimal  if and only if 
$\la X,Z\ra =0$  and  $J_Z\subseteq J_X$, i.e., if $y_i>0$  for  some $i\in J$ then  $\la A_i,X\ra=b_i$. We refer to these two conditions as the {\em complementary slackness} conditions. 

The following theorem provides  an explicit characterization of the space of perturbations of an element of the primal feasible region~$\mcp$.

\begin{theorem}\cite{LT,DL} Consider  
%\begin{equation}\label{specP}
%\mcp=\{X \psd 0 : \la A_i,X\ra =b_i,\  i\in [m]\}.
%\end{equation}
a matrix   $X\in \mcp$, written as  $X=PP^\sfT$, where $P\in \oR^{n\times r}$ and $r=\rankspace X$.
 Then,
\begin{equation}\label{eq:pert}
\per_{\mcp}(X)=\left\{PRP^\sfT : R\in \mcs^r,\ \la PRP^\sfT,A_i\ra=0 \ (i \in I\cup J_X)  \right\}.
\end{equation}
%Moreover, 
%\begin{equation}\label{eq:dimpert}
%\dim \mF_{\mcp}(X)=\binom{r+1}{2}-\dim\la P^\sfTA_iP: i\in [m]\ra.
%\end{equation}
\end{theorem}

%The next theorem provides a characterization of extreme points of spectrahedra.
 As a direct application, we obtain a characterization for extreme points of the primal feasible region~$\mcp$.
 
\begin{corollary}\label{thm:conic} 
Consider  a matrix $X \in \mcp$, written as $X=PP^\sfT$, where  $P\in \oR^{n\times r}$ and $r=\rankspace X$.
 The following assertions are equivalent:
\begin{itemize}
\item[(i)] $X$ is an extreme point of $\mcp$.
\item[(ii)]  If $R\in \mcs^r$ satisfies $\la P^\sfT A_iP, R\ra =0$ for all  $i \in I\cup J_X$, then $R=0$.
\item[(iii)] ${\rm lin}\{P^\sfT A_iP:  i\in I\cup J_X \}=\mcs^r.$
\end{itemize}
\end{corollary}

%\begin{proof} The equivalence $(ii) \Longleftrightarrow (iii)$ is immediate.  To show $(i) \Longleftrightarrow (ii)$ consider $\per_{\mcp}(X)=\{PRP^\sfT : \la PRP^\sfT , A_i\ra =0\  \forall i\in [m], \ R \in \mcs^d\}$. As $\la PRP^\sfT , A_i\ra=\la P^\sfTA_iP,R\ra$ the claim follows.\qed
%\end{proof}

 We denote  by $\mcr_r$  the manifold  of 
   symmetric $n \times n$ matrices with  rank equal  to $r$.  
 Given  a matrix   $X\in \mcr_r$,  let   $X=Q\Lambda Q^\sfT$ be its  spectral decomposition, where $Q$ is an  orthogonal matrix whose columns are the eigenvectors of $X$    and $\Lambda$ is the  diagonal matrix with the corresponding  eigenvalues as diagonal entries. Without loss of generality  we may assume  that $\Lambda_{ii}\ne 0$ for $i\in [r]$.% and let $Q=\begin{bmatrix}R & N
% \end{bmatrix}$ where $R \in \mcs^{n\times r}$ is a matrix whose columns form an orthonormal  basis of $\ran X$ and $N \in \mcs^{n\times (n-r)}$ is a matrix whose   columns  form an orthonormal basis of  $\Ker  X$. 
 
\medskip  
 The {\em tangent space} of $\mcr_r$ at $X$ is given by %\footnote{Do we also want the tangent space to $Z$?}

\begin{equation}
\mct_X=\left\{Q \left(\begin{matrix}
U & V\\
V^\sfT & 0\end{matrix}\right)Q^\sfT: U \in \mcs^r, V\in \mr^{r \times (n-r)}\right\}.
\end{equation}
%The {\em normal space} of $\mcr_r$ at $X$ is given by 
Hence, its orthogonal complement is defined by 
\begin{equation}\label{eq:normal}
\mct_X^{\perp}=\left\{ Q \left(\begin{matrix}
0 & 0\\
0 & W\end{matrix}\right)Q^\sfT: W \in \mcs^{n-r}\right\}. %=\left\{M \in \mcs^n : MX=0\right\}.
\end{equation}
We will also use the equivalent description: 
\begin{equation}\label{eq:normal2}
\mct_X^{\perp}=\{ M \in \mcs^n: XM=0\}.
\end{equation}

%\begin{proof}The first  inclusion follows directly from~\eqref{eq:normal}. For the other inclusion consider a matrix $M \in \mcs^n$ such that  $XM=0$. Then, there exists an $n \times n$ matrix $\Lambda $ such that  $M=N \Lambda$ and thus  $X\Lambda^{\sfT} N^{\sfT}$=0.  As $N^{\sfT}$ has full row rank it follows that $X\Lambda^{\sfT}=0$ and thus there exists a $n\times n$ matrix $\Lambda'$ such that $\Lambda^{\sfT}=N \Lambda'$.  Combining the above we get that $M=N\Lambda'N^{\sfT}$ which implies~\eqref{eq:normal2}. 
%\end{proof}
 
  We now introduce the notions of nondegeneracy and strict complementarity  for the  semidefinite programs (\ref{eq:primalsdp}) and (\ref{eq:dualsdp}) in standard form.% with only equality constraints (i.e., $J=\emptyset$). {\bf TO DO}
 
\begin{definition}\cite{AHO}\label{def:nondeg} Consider the  pair of primal and dual semidefinite programs %Let $\mcp$ and $\mcd$ be the feasible regions of the primal and dual programs 
(\ref{eq:primalsdp}) and (\ref{eq:dualsdp}). 
A matrix $X\in \mcp $  is  called {\em primal nondegenerate} if 
\begin{equation}\label{eq:nondeg}
 \mct_X+\lin\{ A_i : i\in I\cup J_X \}^{\perp}=\mcs^n.  
 \end{equation}
The pair $(y,Z) \in \mcd$ is  called {\em dual nondegenerate} if 
\begin{equation}\label{eq:dnondeg}
\mct_Z+\lin \{ A_i : i \in I\cup J_Z\} =\mcs^n.
\end{equation}
\end{definition}

Next we present  some  well known  results  that  provide  necessary and sufficient conditions for the unicity of  optimal solutions in terms of the notions of primal or dual nondegeneracy and strict complementarity.
% In both results we assume that the primal and dual programs (\ref{eq:primalsdp}) and (\ref{eq:dualsdp}) have optimal solutions with no duality gap.
With the intention to make the section self-contained we have also included short proofs.

\begin{theorem}\label{thm:uprimal}\cite{AHO}\label{thm:uniqueprimal} 
%Consider a pair of primal dual semidefinite programs as in 
Assume that the optimal values of~\eqref{eq:primalsdp} and~\eqref{eq:dualsdp} are equal and that both are attained. 
If~\eqref{eq:primalsdp} has a   nondegenerate optimal solution, then~\eqref{eq:dualsdp} has a   unique optimal solution.
(Analogously, if~\eqref{eq:dualsdp} has  a   nondegenerate optimal solution, then~\eqref{eq:primalsdp} has a unique optimal  solution.)
\end{theorem}

\begin{proof}Let $X$ be a nondegenerate optimal solution of~\eqref{eq:primalsdp} and let $(y^{(1)},Z_1),$ $(y^{(2)},Z_2)$ be two dual optimal solutions. Complementary slackness implies that   $y^{(1)}_j=y^{(2)}_j=0$ holds for every $i\in J\setminus J_X$. Hence,   $Z_1-Z_2\in \lin \{A_i: i \in I\cup J_X\} $.  As there is no duality gap we have  that $XZ_1=XZ_2=0$ and then~\eqref{eq:normal2} implies that $Z_1-Z_2\in \mct_{X}^{\perp}$. These two facts combined  with   the assumption that   $X$ is primal  nondegenerate imply   that  $Z_1=Z_2$. The other case is similar.\qed
\end{proof}
\medskip 

The next  lemma provides a characterization of  the space of perturbations in terms of tangent spaces for a pair of strict complementary optimal solutions. 

\begin{lemma}\label{lem:per} Assume that the optimal values of~\eqref{eq:primalsdp} and~\eqref{eq:dualsdp} are equal and that both are attained. Let $X, (y,Z)$ be a strict complementary pair of primal and dual optimal solutions for \eqref{eq:primalsdp} and \eqref{eq:dualsdp}, respectively.  Then, 
\begin{equation}\label{pert1}
 \per_{\mcp}(X)=\lin\{ A_i: i \in I\cup J_X\} ^{\perp}\cap \mct_Z^{\perp},
 \end{equation}
\begin{equation}\label{pert2} 
 \per_{\mcd}(Z)=\lin \{ A_i: i \in I\cup J_Z\}^{\perp}\cap \mct_X^{\perp}.
 \end{equation}
\end{lemma}

\begin{proof}By assumption,    $ZX=XZ=0$,  which implies that $X$ and $Z$ can be simultaneously diagonalized by the same   orthogonal matrix $Q$.  Let $r=\rankspace X$ and write $Q=(Q_1\ Q_2)$, where the columns of $Q_1\in \mr^{n\times r}$ form a basis of the range of $X$. As $X$ and $Z$ are strict complementary we obtain that  
$$X=Q\left(\begin{matrix}
\Lambda_1 & 0 \\
0 & 0
\end{matrix}\right)Q^\sfT =Q_1\Lambda_1 Q_1^\sfT,\
Z=Q\left(\begin{matrix}
0 & 0 \\
0 & \Lambda_2
\end{matrix}\right)Q^\sfT=Q_2\Lambda_2 Q_2^\sfT,
$$
where $\Lambda_1$ and $\Lambda_2$ are diagonal matrices of sizes $r$ and $n-r$,  respectively.  The claim  follows easily  using  the form   of $\mct_X$ (and $\mct_Z$)  given in \eqref{eq:normal}.\qed
\end{proof}

\medskip 
%For the last two theorems of this section we  additionally  need to  assume that $J=\emptyset$, i.e., the  semidefinite program~\eqref{eq:primalsdp} is defined {\em only by equality constraints}. It is open to determine whether the theorems remain valid  in the presence of inequalities. 

 The next theorem establishes the converse of Theorem~\ref{thm:uprimal}, assuming strict complementarity. 
 %Its proof is a straightforward application  of  Lemma~\ref{lem:per}. 

\begin{theorem}\label{thm:converse}\cite{AHO}\label{thm:converse} 
Assume that the optimal values of~\eqref{eq:primalsdp} and~\eqref{eq:dualsdp} are equal and that both are attained. 
%Consider a pair of primal dual semidefinite programs as in \eqref{eq:primalsdp} and \eqref{eq:dualsdp} and 
Let $X, (y,Z)$ be a strict complementary pair of  optimal solutions for \eqref{eq:primalsdp} and \eqref{eq:dualsdp}, respectively, and assume that   $J_X=J_Z$. If $X$ is the unique optimal solution of~\eqref{eq:primalsdp}
 then $(y,Z)$ is   dual nondegenerate.  (Analogously, if $(y,Z)$ is the unique optimal solution of~\eqref{eq:dualsdp}   then $X$ is primal nondegenerate.)
\end{theorem}

\begin{proof}By assumption, $X$ is the unique optimal solution of (\ref{eq:primalsdp}). Hence $X$   is an extreme point of the primal feasible region and thus, using (\ref{pert1}),
we obtain that $\mct_Z + \lin\{A_i: i\in I\cup J_X\}=\mcs^n$.
As $J_X=J_Z$, (\ref{eq:dnondeg}) holds and thus $(y,Z)$ is dual nondegenerate.
\qed \end{proof}

 \if 0 \begin{proof}By assumption perfect duality holds and as $X, (y,Z)$ are primal dual optimal it follows that  $X$  and $Z$ can be simultaneously diagonalized by some    orthogonal matrix $Q$.  Let $r=\rankspace X, s=\rankspace Z $ and write $Q=[Q_1\ Q_2]$, where the columns of $Q_1\in \mr^{n\times r}$ form a basis for  $\ran X$ and the columns of $Q_2 \in \mr^{n\times s}$ form a basis for $\ran Z$.  As $X$ and $Z$ are strict complementary we obtain that  $X=Q_1\Lambda_1 Q_1^{\sfT}$ and $Z=Q_2\Lambda_2Q_2^{\sfT}$. Let us assume for contradiction that $(y,Z)$ is dual degenerate, i.e., $\mct_{Z}^{\perp}\cap \lin\{A_1,\ldots,A_m\}\not=\{0\}$. 

Let $X'$ be  a primal optimal solution. By assumption perfect duality holds and thus $X'Z=0$ which implies that $X'=Q_1UQ_1^{\sfT}$ for some $U \in \mcs^r$. Then primal feasibility reduces to 
\begin{equation}\label{lem}
\la A_i,Q_1UQ_1^{\sfT} \ra=b_i \ \forall i\in [m].
\end{equation}
As the elements of $\mct_{Z}^{\perp}$ have the form $Q_1UQ_1^{\sfT}$, the fact that $(y,Z)$ is dual degenerate implies that the homogeneous version of~\eqref{lem} has a non trivial solution and thus~\eqref{lem} has as solutions some affine subspace $W$  of positive dimension. But $\Lambda_1\in W$  and it is positive definite so close to $\Lambda_1$ we can find matrices $U$ psd that give solutions which are optimal.\qed
\end{proof}
\fi

%The next theorem follows easily by combining Lemma~\ref{lem:per} with Theorems  \ref{thm:uprimal} and  \ref{thm:converse}.
\medskip
As an application we obtain the following characterization for the extreme points of $\mcp$, assuming strict complementarity.

\begin{theorem}\label{main}Assume that the optimal values of~\eqref{eq:primalsdp} and~\eqref{eq:dualsdp} are equal and that both are attained. 
Let $X, (y,Z)$ be a pair of strict complementary optimal solutions of the primal  and dual programs \eqref{eq:primalsdp} and \eqref{eq:dualsdp}, respectively, and assume that  $J_X=J_Z$.
%Consider  a pair of primal-dual semidefinite programs as in \eqref{eq:primalsdp} and \eqref{eq:dualsdp}.  Suppose there exists  a pair $X,(y,Z)$ of primal-dual optimal solutions satisfying strict complementarity.
%Assume that the primal is feasible and let $X^*\in \mcp$. Moreover, assume there exists a   matrix $Z^*\in \mcd$ such that $\la X^*,Z^*\ra =0$ and $\rankspace X^*=\cor Z^*$. 
The following assertions are equivalent:
\begin{itemize}
\item[(i)] $X$ is an extreme point of $\mcp$. 
\item[(ii)] $X$ is the unique primal optimal solution of \eqref{eq:primalsdp}.
\item[(iii)] $Z$ is a dual nondegenerate. %solution of \eqref{eq:dualsdp}.
\end{itemize}
\end{theorem}

\begin{proof}The equivalence $(ii) \Longleftrightarrow (iii)$  follows directly from     Theorems \ref{thm:uprimal} and \ref{thm:converse} and  the equivalence $(i)  \Longleftrightarrow (iii)$ follows by Lemma~\ref{lem:per} and the definition of dual nondegeneracy from \eqref{eq:dnondeg}.\qed
\end{proof}

\medskip 
Note that  Theorems \ref{thm:converse} and \ref{main} still hold if we replace the condition $J_X=J_Z$ by the weaker condition:
\begin{equation}\label{eq:weak}
\forall i\in J_X\setminus J_Z\ \ A_i\in \mct_Z+\lin\{A_i: i\in I\cup J_X\}.
\end{equation}
Note also that this condition is automatically satisfied in the  case when $J=\emptyset$, i.e., when the  semidefinite program~\eqref{eq:primalsdp} involves only linear equalities.
\section{Uniqueness  of positive semidefinite  matrix completions}\label{secGram}

\subsection{Basic definitions}

%Recall that, given vectors $p_1,\ldots,p_n \in \oR^k$ ($k\ge 1$), ${\rm Gram}(p_1,\ldots,p_n)=(p_i^{\sfT}p_j)_{i,j=1}^n$ denotes  their  {\em Gram matrix}, whose rank is equal to the dimension of theor linear span.
% It is well known that for a matrix   $A\in \PSD$ with  $\rankspace A=d$
% there exist vectors $p_1,\ldots,p_n\in \oR^d$  such that $A={\rm Gram}(p_1,\ldots,p_n)$. Moreover, if $A={\rm Gram}(q_1,\ldots,q_n)$ for some other vectors  $q_1,\ldots,q_n$ then there exists a $d \times d$ orthogonal matrix $O$ such that $p_i=Oq_i$ for all $i \in [n]$.

Let $G=(V=[n],E)$ be a given graph.  Recall that a vector $a\in \oR^{V\cup E}$ is called a {\em $G$-partial psd matrix} if it admits at least one completion to a full psd matrix, %$A\in \mcs^n_+$
 i.e., if the semidefinite program (\ref{eq:psdcompl}) has at least one feasible solution. 
We denote by  $\mcs_+(G)$  the set of all $G$-partial psd matrices. In other words, $\mcs_+(G)$ is equal to the projection of the positive semidefinite cone $\mcs^n_+$ onto the subspace $\oR^{V\cup E}$ indexed by the nodes (corresponding to the diagonal entries) and the edges of $G$.
We can reinterpret $G$-partial psd matrices in terms of Gram representations. Namely, $a\in \mcs_+(G)$ if and only if there exist vectors $p_1,\ldots,p_n\in \oR^d$ (for some $d\ge 1$) such that $a_{ij}=p_i^\sfT p_j$ for all $\{i,j\}\in V\cup E$. This leads to the notion of frameworks which will make the link between the Gram (spherical) setting of this section and the Euclidean distance setting considered in the next section.

%To motivate the  definitions that follow, consider a graph $G=([n],E)$ and  a  partial matrix $a \in \ps$  which admits a unique completion to a full psd matrix  $A \in \PSD $ where  say $d=\rankspace A$. 
%Thinking in terms of the Gram representation of $A$ we can reinterpret the  condition that $a \in \ps$ admits a unique psd completion as follows: For any family of vectors $q_1,\ldots,q_n$ having the property that $q_i^{\sfT}q_j=p_i^{\sfT}p_j$ for all $\{i,j\}\in V\cup E$ there exists an orthogonal $d \times d$  matrix  $O$ such that $q_i=Op_i$ for all $i\in [n]$. This observation already hints to a connection to the universal rigidity of frameworks which we will make precise in Section~\ref{sec:connections}. 

A {\em  tensegrity  graph} is a graph $G$  whose edge set is partitioned into three sets: $E=B\cup C\cup S$, whose members are called   {\em bars, cables} and {\em struts}, respectively.   A {\em  tensegrity framework} $G({\bf p})$  consists of  a tensegrity graph  $G$ together with an  assignment of vectors ${\bf p}=\{p_1,\ldots,p_n\}$ to the nodes of $G$. A {\em bar framework} is a tensegrity framework where $C=S=\emptyset$.

\if 0 
\begin{definition} Let $G=([n],E) $ be a tensegrity graph with $E=B \cup C\cup~S$. A  tensegrity  framework $G(\p)$  is said to {\em spherically dominate} a  tensegrity framework $G(\q)$ if the following conditions hold:
\begin{itemize}
\item[(i)] $\|p_i\|=\|q_i\|\ $ for all $i\in [n]$,
\item[(ii)]$ p_i^{\sfT}p_j=q_i^{\sfT}q_j \ $  for all (bars) $ \{i,j\} \in B$,
\item[(iii)]$ p_i^{\sfT}p_j\ge q_i^{\sfT}q_j\ $ for all (cables) $\{i,j\}\in C$,
\item[(iv)]$ p_i^{\sfT}p_j\le q_i^{\sfT}q_j\ $ for all (struts) $\{i,j\} \in S$.
\end{itemize}
\end{definition}
The next definition is the analogue of the notion of congruent  frameworks. 
\begin{definition}
Two frameworks  $G(\p)$ and $G({\bf q})$  are  called {\em spherically congruent} if 
$$ p_i^{\sfT}p_j=q_i^{\sfT}q_j  \ \text{ for all } i,j  \in [n].$$ 
\end{definition}
\fi

Given a tensegrity framework $G(\p)$ consider the following pair of primal and dual semidefinite programs:

\begin{equation}\tag{$\mcp_G$}\label{eq:ssdp} 
\begin{array}{ll}
%\underset{X}
{\sup}_X \{  0 : X\succeq 0, & \la E_{ij},X\ra=p_i^ \sfT p_j\  \text{ for } \{i,j\} \in V\cup B,\\
& \la E_{ij},X\ra\le p_i^ \sfT p_j \ \text{ for } \{i,j\} \in C,\\

& \la E_{ij},X\ra\ge p_i^ \sfT p_j \ \text{ for } \{i,j\} \in S 
\}
\end{array}
\end{equation}

and 
\begin{equation}\tag{$\mcd_G$} \label{eq:ssdpd}
\begin{array}{ll}
%\underset{y,Z}
{\inf}_{y,Z} \{ \sum_{ij \in V\cup E} y_{ij}p_i^\sfT p_j : & \sum_{ij \in V\cup E}y_{ij}E_{ij}=Z \psd 0,\\ 
 & y_{ij}\ge 0 \  \text{ for } \{i,j\}  \in C,\\
 &  y_{ij}\le 0\  \text{ for } \{i,j\}  \in S \}.
\end{array} 
\end{equation}

The next   definition captures  the analogue of the notion of universal rigidity for  the  Gram setting.

\begin{definition} A tensegrity framework $G(\p)$ is called {\em universally completable}  if the matrix $\gram(p_1,\ldots,p_n)$ is the unique solution  of the semidefinite program~\eqref{eq:ssdp}. 

\end{definition}
%This definition is relevant for our purposes since 
In other words,   a universally completable framework  $G(\p)$ corresponds to a   $G$-partial psd matrix $a\in \mcs_+(G)$, where $a_{ij}=p_i^\sfT p_j$ for all $ \{i,j\} \in V\cup E$,    that admits a unique completion to a full psd matrix. %feasible for \eqref{eq:ssdp}.
%, given by ${\rm Gram}(p_1,\ldots,p_n)$.  
Consequently, identifying sufficient conditions  guaranteeing  that a framework $G(\p)$ is universally completable will allow us to construct $G$-partial matrices with a unique psd completion.

\subsection{A sufficient  condition for universal completability}

In this section we  derive  a sufficient condition  for determining the universal  completability of tensegrity frameworks. 

We use the following notation: For a graph $G=(V,E)$, $\overline E$ denotes the set of pairs $\{i,j\}$ with $i\ne j$ and $\{i,j\}\not\in E$, corresponding to the non-edges of $G$. % (i.e., the edges of the complementary graph $\overline G$).

\begin{theorem}\label{thm:sphcorp}
Let $G=([n],E)$ be a tensegrity graph with $E=B\cup C\cup S$ and consider a tensegrity framework   $G({\bf p})$ in $\oR^d$ such that   $p_1,\ldots,p_n$ span linearly $\oR^d$.  Assume  there exists a matrix $Z\in \mcs^n$ satisfying the conditions (i)-(vi):
\begin{itemize}
\item[(i)] $Z$ is positive semidefinite.
\item[(ii)]
$Z_{ij}=0$ for all $\{i,j\}\in \overline E$. 
\item[(iii)] $Z_{ij}\ge 0$ for all (cables) $\{i,j\} \in C$  and $Z_{ij}\le 0$ for all (struts) $ \{i,j\} \in S.$
\item[(iv)] $Z$ has corank $d$.
\item[(v)] $\sum_{j\in V}Z_{ij}p_j=0$ for all $i \in [n]$.
%\la Z, \gram(p_1,\ldots,p_n)\ra =0$ (or, equivalently, $Z\ \gram(p_1,\ldots,p_n)=0$). 
\item[(vi)]
%Moreover, assume that $G(\p)$  satisfies the following condition:
For any matrix $R\in \mcs^d$ the following holds: 
\begin{equation}\label{eq:sphconic}
 p_i^\sfT R\  p_j=0\   \forall  \{ i,j\} \in V\cup B \cup \{\{i,j\}\in C\cup S: Z_{ij}\ne 0\}   \Longrightarrow  R=0.
 \end{equation}
\end{itemize}
Then the tensegrity  framework $G({\bf p})$ is universally  completable.
 %, i.e., the partial matrix $a_{\bf p}=(p_i^\sfT p_j)_{\{i,j\}\in V\cup E}$ has a unique psd completion.
\end{theorem}

\begin{proof}Set $X={\rm Gram}(p_1,\ldots,p_n)$.  Assume that   $Y\in \mcs^n_+$  is another matrix which is feasible for the program \eqref{eq:ssdp}, say    $Y=\gram(q_1,\ldots,q_n)$ for some vectors $q_1,\ldots,q_n$. Our goal is to show that $Y=X$. 
  By   $(v)$, $ZX=0$     and thus $\ran X \subseteq \Ker Z$. Moreover, $\dim \Ker Z=d$ by (iv), and $\text{rank} X=d$ since  $\lin\{p_1,\ldots,p_n\}=\oR^d$.
  %,  we have that $\rankspace X=d$. This fact,  combined with $(iv)$,  
  This implies that $\Ker X= \ran Z$. 
  
By (ii) we can write $Z=\sum_{\{i,j\}\in V\cup E} Z_{ij} E_{ij}$.   Next notice that 
\begin{equation}\label{eq:2}
 0 \le \la Z,Y\ra=\left\la \sum_{\{i,j\} \in V\cup E}Z_{ij}E_{ij},Y\right\ra \le  \sum_{\{i,j\} \in V\cup E}Z_{ij}\la E_{ij},X\ra=\la Z,X\ra=0,
 \end{equation}
 where the first (left most) inequality follows from the fact that $Y,Z\succeq 0$ and the second one from the feasibility of $Y$ for \eqref{eq:ssdp} and the sign conditions (iii) on $Z$.
This gives   $\la Z,Y\ra=0$,  which implies that $\Ker Y \supseteq \ran Z$ and thus $\Ker Y  \supseteq \Ker X$. 

Write $X=PP^\sfT$, where $P\in \oR^{n\times d}$ has rows $p_1^\sfT,\ldots,p_n^\sfT$. From the inclusion $\Ker (Y-X)\supseteq \Ker X$,   we deduce that 
$Y-X=PRP^\sfT$ for some matrix $R\in \mcs^d$.
 
As equality holds throughout in   \eqref{eq:2}, we obtain that $\la E_{ij},Y-X\ra =0$ for all $\{i,j\} \in C\cup S$ 
with $Z_{ij}\not=0$. Additionally, as $X,Y$ are both feasible for~\eqref{eq:ssdp},  we have that $\la E_{ij},Y-X\ra=0$ for all $\{i,j\}\in V\cup B$. 
 Substituting $PRP^{\sfT}$ for $Y-X$, we obtain that $p_i^\sfT Rp_j=0$ for all $\{i,j\}\in V\cup B$ and all $\{i,j\}\in C\cup S$ with $Z_{ij}\ne 0$. We can now apply (vi) and conclude that $R=0$. This  gives 
  $Y=X$, which concludes the proof.\qed
\end{proof}

\medskip
Note that the conditions $(i)$-$(iii)$ express that $Z$ is feasible for the dual semidefinite program \eqref{eq:ssdpd}. In analogy to the Euclidean setting (see Section~\ref{sec:con}),   such matrix  $Z$ is called a {\em spherical stress matrix} for the framework $G(\mathbf p)$. 
Moreover, (v) says that $Z$ is dual optimal and (iv) says that $X=\gram(p_1,\ldots,p_n)$ and $Z$ are strictly complementary solutions to the primal and dual semidefinite programs \eqref{eq:ssdp} and \eqref{eq:ssdpd}.
Finally, in the case of bar frameworks (when $C=S=\emptyset$), condition (vi) means that $Z$ is dual nondegenerate. 
Hence, for bar frameworks, Theorem~\ref{thm:sphcorp} also follows as a direct application of Theorem~\ref{main}.

As a last remark notice that the assumptions  of Theorem~\ref{thm:sphcorp} imply that $n\ge d$. Moreover, for $n=d$, the matrix $Z$ is the zero matrix and  in this   case~\eqref{eq:sphconic} reads: $p_i^\sfT Rp_j=0$ for all $\{i,j\} \in V\cup B$ then $R=0$. Observe that this condition can  be satisfied only when $G=K_n$ and $C=S=\emptyset$, so that  Theorem~\ref{thm:sphcorp} is useful only in the case when $d\le n-1$.
%As the vectors $p_1,\ldots,p_n$ are linearly independent  this can  be satisfied only when $G=K_n$ and $C=S=\emptyset$.} 

\subsection{Applying the sufficient condition}\label{sec:examples}

In this section we use Theorem~\ref{thm:sphcorp} to construct  several instances of partial psd matrices admitting  a unique psd completion. Most of these  constructions have been considered in~\cite{ELV,LV,LV12}. While  the proofs there for  unicity of the psd completion   consisted of ad hoc arguments and case checking, Theorem~\ref{thm:sphcorp} provides  us  with a unified and systematic approach. In all examples below we only deal with bar frameworks and hence  we apply Theorem \ref{thm:sphcorp} with $C=S=\emptyset$. In particular, there are no sign conditions on the stress matrix $Z$ and  moreover  condition (\ref{eq:sphconic}) assumes  the simpler form: 
If $p_i^\sfT Rp_j=0$ for all  $\{i,j\}\in V\cup E$  then  $R=0$.  
%  and allows us to recover most of the instances from~\cite{LV,ELV}. Lastly, we wish  to point to the fact   that in all the  constructions that follow  we apply Theorem~\ref{thm:sphcorp} with $C=S=\emptyset$.

%\begin{example} 
%\subsubsection*{Example 1: The octahedron}
\medskip\noindent
{\bf Example 1: The octahedral graph.}
%Let $G=K_{2,2,2}$ be   the octahedral graph illustrated in Figure~\ref{oct}.
Consider a framework  for the octahedral  graph $K_{2,2,2}$ defined as follows:
$$p_1=e_1,\  p_2=e_2,\ p_3=e_1+e_2,\  p_4=e_3,\  p_5=e_4,\  p_6=e_5,$$ 
where  $e_i\ (i \in [5])$ denote the standard unit vectors in $\oR^5$ and the numbering of the nodes refers to Figure 1.
%Let $A=(a_{ij})\in \mcs^6_+$ be the Gram matrix of the vectors $p_1,\ldots,p_6$. 
In~\cite{LV} it is shown that the corresponding  $K_{2,2,2}$-partial matrix $a=(p_i^\sfT p_j) \in \mcs_+(K_{2,2,2})$ %where $a_{ij}=p_i^\sfT p_j \ ({\{i,j\} \in V\cup E})$ 
admits  a unique psd completion.
This result follows easily, using  Theorem~\ref{thm:sphcorp}.
Indeed it is easy to check that  condition (\ref{eq:sphconic}) holds. Moreover, the matrix 
%using Lemma \ref{lem:sphconic}, one can easily verify that $A$  is an extreme point of  the feasible region of (\ref{eq:sdp2}).
    $Z=(1,1,-1,0,0,0)(1,1,-1,0,0,0)^\sfT$ is psd with corank 5, it is supported by $K_{2,2,2}$, and satisfies $\la Z, \gram(p_1,\ldots,p_5)\ra=0$.
%    
%    is feasible for (\ref{eq:sdp2d})  and satisfies  $ZA=0$ and 
%$\cor Z=5=\rankspace A$. 
Hence Theorem~\ref{thm:sphcorp} applies  and the claim follows. 

\begin{figure}[h]
\centering \includegraphics[scale=0.4]{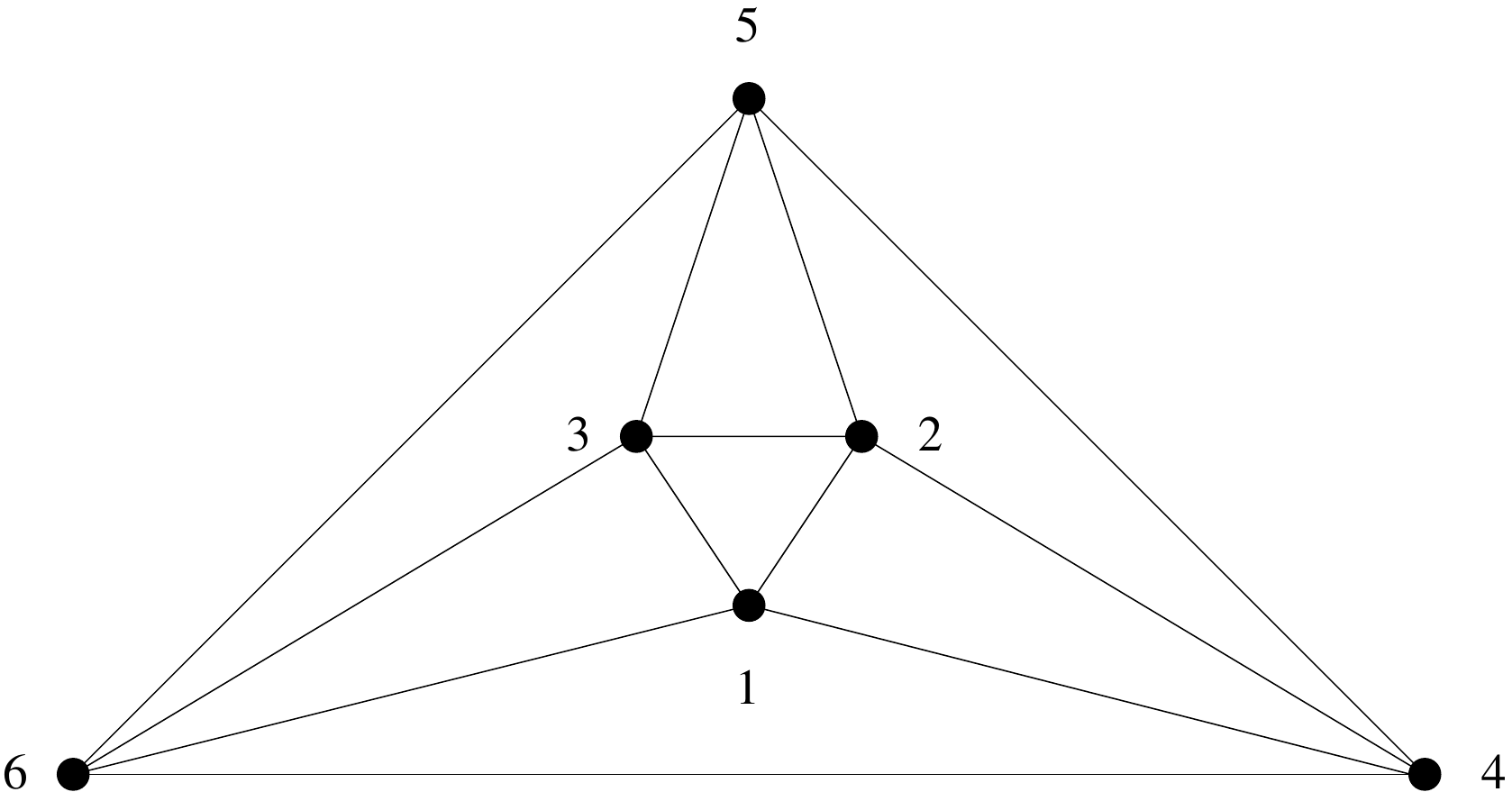}
\label{oct}
\caption{The graph $K_{2,2,2}$.}
\end{figure}
%\end{example}

%\begin{example}
\medskip\noindent
{\bf Example 2: The family of graphs $F_r$.}
For an integer $r \ge 2$,  we define a   graph $F_r=(V_r,E_r)$ with  $r+\binom{r}{2}$ nodes denoted as $v_i$ (for $i \in [r])$ and $v_{ij}$ (for $1\le i<j\le r$). It consists of a central clique of size $r$ based on the nodes $v_1,\ldots,v_r$ together with the cliques $C_{ij}$ on the nodes $\{v_i,v_j,v_{ij}\}$. The graphs $F_3$ and $F_4$ are shown in Figure~2 below. 
\begin{figure}[h]
\centering \includegraphics[scale=0.45]{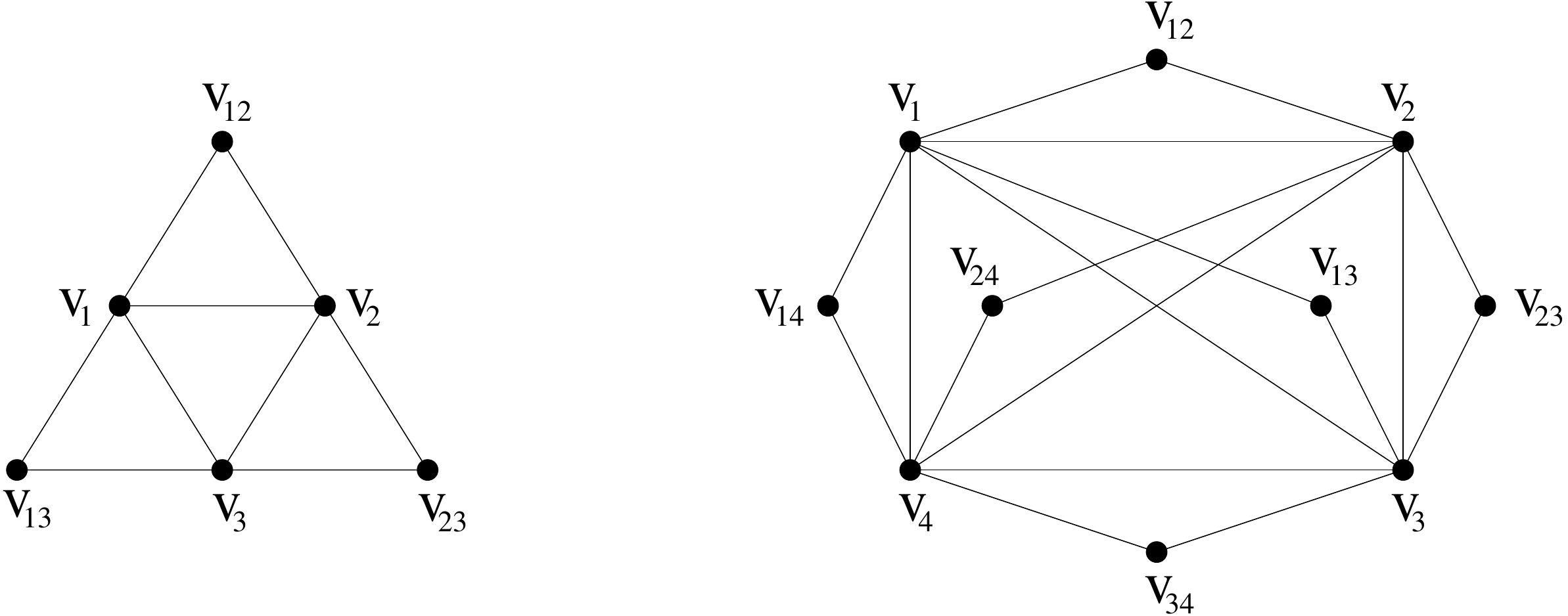}\label{Fr}
\caption{The graphs $F_3$ and $F_4$.}
\end{figure}
We construct a  framework in $\oR^r$  for the graph $F_r$  as follows:
$$p_{v_i}=e_i \text{ for } i\in [r] \text{ and } p_{v_{ij}}=e_i+e_j \text{ for } 1\le i<j\le r.$$
%where by $e_i \ (i\in [r])$ we denote the standard unit vectors in $\oR^r$.  
%we map the nodes in the central clique   to the standard unit  vectors in $\mr^r$  and every other  node $v_{ij}$  to the vector $e_i+e_j$. 
In~\cite{ELV} it is shown that for any $r\ge2$  the corresponding  $F_r$-partial matrix  %$a\in \mcs_+(F_r)$ %where   $a_{ij}=p_i^\sfT p_j\ (\{i,j\} \in V_r\cup F_r)$ 
admits a unique psd completion.
% In~\cite{ELV} it is shown that  the associated partial matrix $a_{\bf p}$ has   a unique positive semidefinite  completion. 
We now show this result,  using Theorem~\ref{thm:sphcorp}.

%Using Lemma \ref{lem:sphconic}, it is easy to verify that $A$ is an extreme point of  the feasible region of (\ref{eq:sdp2}). We now construct a dual optimal solution $Z$ to (\ref{eq:sdp2d}) strictly complementary to $A$.
Fix $r\ge 2$. It is easy to check  that (\ref{eq:sphconic}) holds. %We now define a matrix $Z$ satisfying the conditions (i),(ii),(ii) of Theorem~\ref{thm:sphcorp}.
Define the nonzero  matrix $Z_r=\sum_{1\le i<j\le r} u_{ij}{u_{ij}}^\sfT$, where the vectors  $u_{ij}\in \mr^{r+\binom{r}{2}}$ are defined as follows: 
For $1\le k\le r$,   $(u_{ij})_k=1$ if $k\in \{i,j\}$ and 0 otherwise; for $1\le k< l\le r$,    $(u_{ij})_{kl}=-1$ if $\{k,l\}=\{i,j\}$ and 0 otherwise.   
%Let $Z_r=\sum_{1\le i<j\le r }u_{ij}u_{ij}^\sfT $ and notice that by construction $Z_r$ is a positive semidefinite equilibrium stress matrix for $F_r(\p)$.
By construction, $Z_r$ is psd, it is supported by the graph $F_r$, $\la Z_r, \gram(p_v : v\in V_r)\ra =0$ and  $\corank Z_r=r$. Thus Theorem~\ref{thm:sphcorp} applies  and the claim follows. 

%$. s the vectors $u_{ij}$ are linearly independent we have that $\rankspace Z_r =\binom{r}{2}$ and thus  $\cor Zr=r=\rankspace A_r$. 
%\end{example}

%\begin{question}In \cite{M} the following question is asked:\\
%{\em If a rational partial psd matrix has a psd completion, does a rational psd completion always exist ? }\\
%The simplest case for which this is not known is the case of $C_5$. Can we use the main theorem to construct a rational $C_5$-partial matrix with a unique completion with at least one irrational entry?
%\end{question}

\medskip\noindent{\bf Example 3: The family of graphs $G_r$.} 
This family of graphs has been considered  in the study of the Colin de Verdi\`ere graph parameter~\cite{CdV98}.  For any integer $r\ge 2$ consider an equilateral triangle and subdivide each side into $r-1$ equal segments. Through these points draw line segments  parallel to the sides of the triangle. This construction creates a triangulation of the big triangle into $(r-1)^2$ congruent equilateral triangles.  The graph $G_r=(V_r,E_r)$ corresponds to the edge graph of this triangulation.  Clearly, the  graph $G_r$  has $\binom{r+1}{2}$ vertices, which we denote  $(i,l)$ for $l\in [r]$ and $i\in [r-l+1]$. For any fixed $l\in [r]$ we say that the  vertices $(1,l),\ldots,(r-l+1,l)$ are at level $l$.
Note that $G_2=K_3=F_2$, $G_3=F_3$, but $G_r\ne F_r$ for  $r\ge 4$. The graph $G_5$ is illustrated in Figure~\ref{fig:t5}. 

 \begin{figure}[h]
 \centering
  \includegraphics[scale=0.5]{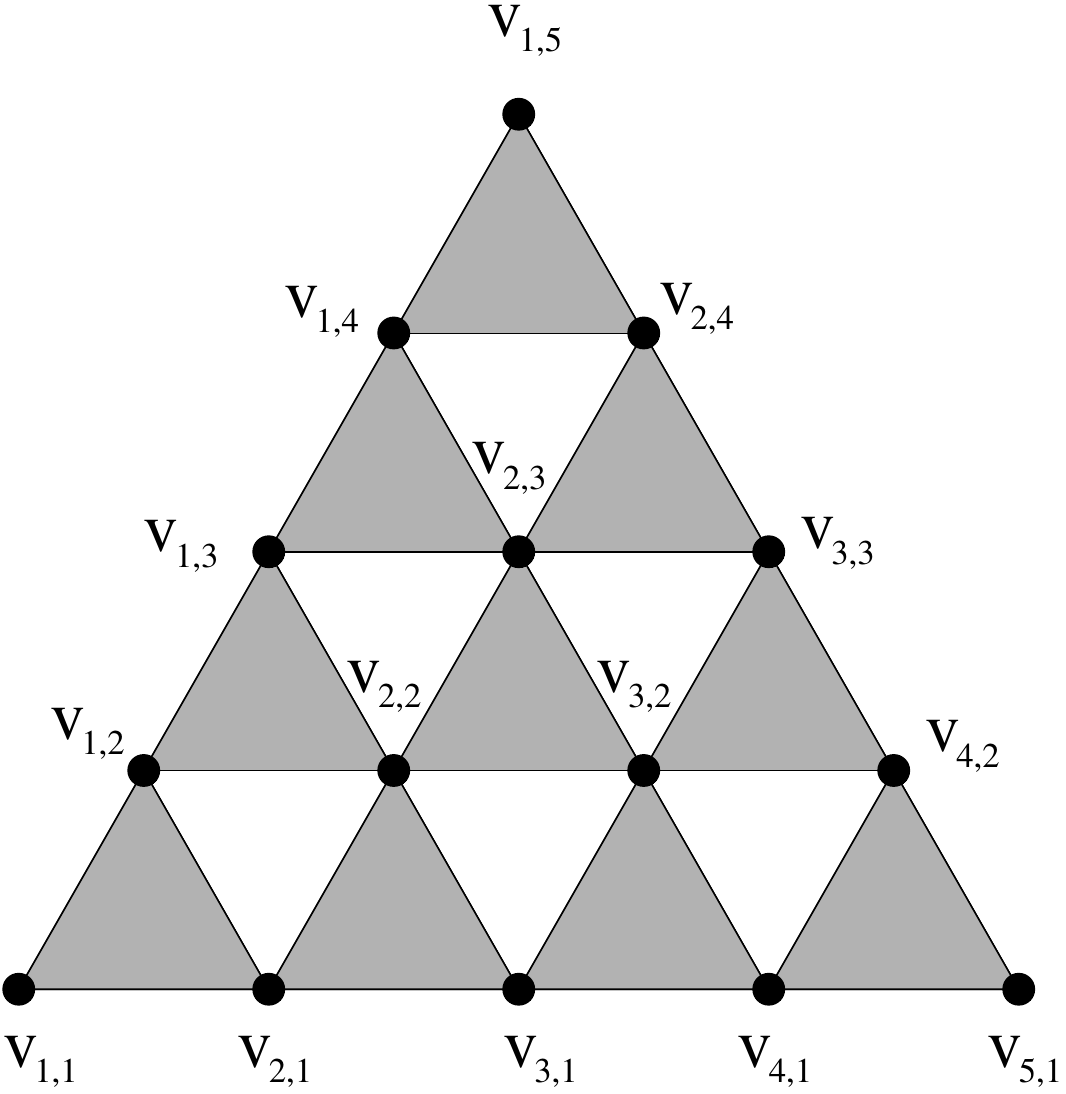}
 \caption{The graph $G_5$.}\label{fig:t5}
 \end{figure}

Fix an integer $r\ge 2$. % and let  $e_i \ (i\in [r])$ denote the standard unit vectors in $\oR^r$. 
We consider the following    framework  in $\oR^r$ for the graph $G_r$:
\begin{equation}\label{frame}
p_{(i,1)}=e_i\ \forall  i \in [r] \ \text{ and }
%we map the  vertices  $(1,1),\ldots,(r,1)$  to the standard basis vectors in $\mr^r$. 
%and we proceed inductively by  setting 
p_{(i,l)}=p_{(i,l-1)}+p_{(i+1,l-1)} \ \forall l\ge 2 \text{ and } i\in [r-l+1]. 
\end{equation}
%mapping   vertex $(i,l)$ at level $l$ to the sum of the vectors labeling vertices $(i,l-1)$ and $(i+1,l-1)$ from  the previous  level. 
In~\cite{ELV} it is shown that  for any $r\ge 2$ the partial   $G_r$-partial matrix  that corresponds  to the framework defined in~\eqref{frame} has   a unique psd  completion. 
We now recover  this result,  using Theorem~\ref{thm:sphcorp}.

First  we show that this  framework  satisfies  \eqref{eq:sphconic}. For this, consider a matrix  $R \in \mcs^r$  such that 
$p_{(i,l)}^\sfT R\ p_{(i',l')}=0$ for every $\{(i,l),(i',l')\} \in V_r\cup E_r$. Specializing this relation  for $i'=i\in [r]$ and $l'=l=1$ we get  that $R_{ii}=0$ for all  $i\in [r]$ and for  $i'=i+1$ and $l=l'=1$   we get that  $R_{i,i+1}=0$ for $i\in[r-1]$. 
%Similarly, for $i'=i+1$ and $l'=l=2$  we get that 
 %$R_{i,i+2}=0$ for $i \in [r-2]$ 
Similarly,   for $i'=i+1$ and $l'=l \ge 2$  we get  that $R_{i,i+l}=0$ for all $ i \in [r-l]$ and thus $R=0$.

We call a triangle in $G_r$ {\em black} if it is of the   form $\{ (i,l), (i+1,l), (i,l+1)\}$   and we denote by   $\mathcal{B}_r$ the set of black triangles in $G_r$. The  black triangles in $G_5$ are illustrated in Figure~\ref{fig:t5} as the shaded triangles. Let $Z_r=~\sum_{t \in \mathcal{B}_r}u_tu_t^\sfT$ where the vector $u_t \in \mr^{\binom{r+1}{2}}$ is  defined as follows: If $t \in \mathcal{B}_r$ corresponds  to the black triangle $\{ (i,l), (i+1,l), (i,l+1)\}$ then $u_t(i,l)=u_t(i+1,l)=1,\  u_t(i,l+1)=-1$ and 0 otherwise.  Since $| \mathcal{B}_r|=\binom{r+1}{2}-r$ and the vectors $(u_t)_{ t \in \mathcal{B}_r}$ are linearly independent we have that $\cor Z_r =r$.
Moreover, as  every edge of $G_r$ belongs to exactly one black triangle we have that $Z_r$ is supported by $G_r$.  By construction of the framework we have that  $\sum_{(i,l) \in V_r} p_{(i,l)}u_t=0$ for all $t\in \mathcal{B}_r$ which implies that  $\la \gram(p_{(i,l)}: (i,l)\in V_r), Z_r\ra =0$. Thus Theorem~\ref{thm:sphcorp} applies and the claim follows.

\medskip\noindent{\bf Example 4: Tensor products of graphs.}
This construction  was considered     in~\cite{PV01},  where universally rigid frameworks were   used as a tool to construct uniquely colorable graphs. The original  construction was carried out  in the Euclidean setting for a suspension bar framework. Here we present the construction in the  spherical  setting which, as we will see in Section~\ref{sec:connections}, is equivalent. 

%\begin{example} \cite{PV01} 
 Let $H=([n],E)$ be a $k$-regular graph  satisfying $\underset{2\le i\le n}{\max} | \lambda_i|< k/(r-1)$, where $\lambda_1,\ldots,\lambda_n$ are the eigenvalues of its adjacency matrix $A_H$.
 For $r\in\oN$ we   let $G_r=(V_r,E_r)$ denote  the graph $K_r \times H$, obtained by taking the tensor product of the complete graph $K_r$ and the graph $H$. By construction, the adjacency matrix of $G_r$ is the tensor product of the adjacency matrices of $K_r$ and $H$: $A_{G_r}= A_{K_r}\otimes A_H$.
 Let us denote the vertices of $G_r$ by the pairs $(i,h)$ where $i\in [r]$ and $h \in V(H)$. 
 
Let $w_1,\ldots,w_{r}\in \oR^{r-1}$ be vectors that linearly span $\oR^{r-1}$ and moreover satisfy  $\sum_{i=1}^r w_i=0$.
 We construct a  framework for   $G_r$ in $\oR^r$ by assigning to all nodes $(i,h)$ for $h\in V(H)$ the vector $p_{(i,h)}=w_i$, for each $i\in [r]$.
%Let $A\in \mcs^{kn}_+$ denote the Gram matrix of the framework $\bf p$.
% and let $a$ be the corresponding partial matrix supported by $G_k$.
We now show, using Theorem~\ref{thm:sphcorp},  that  the associated $G_r$-partial matrix   admits a unique psd  completion.

First  we show that this  framework  satisfies~\eqref{eq:sphconic}. For this, consider a matrix $R\in \mcs^r$ satisfying $p_{(i,h)}^\sfT  R  
p_{(i',h')}=0$ for every $\{(i,h),(i',h')\}\in V_r\cup E_r$. This implies that 
$w_i^\sfT Rw_j=0$  for all $i,j\in [r]$ and as $\lin\{w_iw_j^\sfT+w_jw_i^\sfT : i,j \in[r]\}=\mcs^r$ it follows that $R=0$.
%For $i'=i$ and $h'=h$ we get that $w_i^\sfT R w_i =0$ for $i\in [k]$

Next  consider the matrix  $Z_k=I_{rn}+{1 \over k} A_{G_r} \in \mcs^{rn}$, where $I_{rn}$ denotes the identity matrix of size $rn$.  Notice that the matrix $Z_r$ is  by construction supported by $G_r$.
%As the adjacency matrix of the tensor product of two graphs is equal to the tensor product of the corresponding adjacency matrices $\Omega$ has the right pattern. 
One can  verify directly  that $\la \gram(p_{(i,h)}: (i,h) \in V_r), Z\ra=0$.
% (due to the symmetry of the framework it suffices to verify that the  equilibrium condition holds  at one node $(1,h)\in V(G_k)$).
%. Clearly,  $$\sum_{(j,h')\in V(G_k)}\Omega_{(i,h),(j,h')}p_{(j,h')}=\sum_{i=1}^kp_i=0, $$
%and thus $\Omega$ is an equilibrium stress matrix.
The eigenvalues of $A_{K_r}$ are $r-1$ with multiplicity one  and $-1$ with multiplicity $r$. This fact combined with the assumption  on the eigenvalues of $H$  implies   that  $Z_r$ is positive semidefinite with  $\cor Z_r=r-1.$ Thus Theorem~\ref{thm:sphcorp} applies and the claim follows. 
%\end{example}

\medskip\noindent{\bf Example 5: The odd cycle $C_5$.}
The last example  illustrates the fact that sometimes the sufficient conditions from  Theorem~\ref{thm:sphcorp} cannot be used to show existence of a unique psd completion.
Here we consider the 5-cycle  graph $G=C_5$ (although it is easy to generalize the example to arbitrary odd cycles).

First we  consider the framework in $\oR^2$  given by   the vectors $$p_i=(\cos ( { 4(i-1)\pi/ 5}),\sin ({ 4(i-1) \pi / 5}))^\sfT  \ \text{  for } 1\le i\le 5.$$ 
%and the associated Gram matrix $A\in \mcs^5_+$. 
The corresponding  $C_5$-partial matrix   has  a unique psd completion, and this can be shown using  Theorem~\ref{thm:sphcorp}.    %$A$ is an extreme point of the feasible region of (\ref{eq:sdp2}). 

It is easy to see that~(\ref{eq:sphconic}) holds. Let $A_{C_5}$ denote the adjacency matrix of $C_5$  and recall that its eigenvalues are $2\cos { 2 \pi \over 5}$ and $-2\cos { \pi \over 5}$, both with multiplicity two and $2$ with multiplicity one. 
Define  
$Z=2\cos { \pi \over 5}I+A_{C_5}$ and notice that   $Z\succeq 0$ and $\corank Z=2$. Moreover,  one can verify that  $\sum_{j \in [5]} Z_{ij}p_j=0$ for all $i \in [5]$ which implies that  $\la Z, \gram(p_1,\ldots,p_5)\ra=0$. Thus Theorem~\ref{thm:sphcorp} applies and the claim follows.

Next we  consider another framework for $C_5$  in $\oR^2$ given by the vectors 
$$q_1=(1,0)^\sfT, q_2=(-1/\sqrt 2, 1/\sqrt 2)^\sfT, q_3= (0,-1)^\sfT, q_4=(1/\sqrt 2,1/\sqrt 2)^\sfT,$$ $$ q_5=(-1/\sqrt 2,-1/\sqrt 2)^\sfT.$$ 
We now show  that  the corresponding $C_5$-partial matrix admits a unique psd completion. 
%and their associated Gram matrix $A$.  
%It turns out that the associated partial matrix $a_{\bf q}$ has a unique psd completion.
This cannot be shown using Theorem~\ref{thm:sphcorp} 
%For this framework, the corresponding partial matrix $b= (b_{ij})_{\{i,j\}\in V(C_5)\cup E(C_5)}$ also admites a unique psd completion, however this cannot be shown with the help of Corollary \ref{cor:sphcor}. 
since   there does not exist a  nonzero  matrix $Z\in \mcs^5$ supported by $C_5$ satisfying  $\la Z,\gram(q_1,\ldots q_5)\ra=0$. 
Nevertheless one can prove that there exists a unique   psd completion by using the following geometric argument. 

Let $X\in \mcs^5_+$ be a psd completion of the partial matrix and set $\vartheta_{ij}=\arccos X_{ij}\in [0,\pi]$ for $1\le i\le j\le 5$. Then,
$\vartheta_{12}=\vartheta_{23}=\vartheta_{34} =\vartheta _{45}= 3\pi/4$ and $\vartheta_{15}=\pi$. Therefore,  the following  linear equality holds:

\begin{equation}\label{eqangle1}
\sum_{i=1}^5 \vartheta_{i,i+1} = 4\pi
\end{equation}
(where indices are taken modulo 5).
As we will see this implies that the remaining angles are uniquely determined by the relations:
\begin{equation}\label{eqangle2}
\vartheta_{i,i+2}+ \vartheta_{i,i+1}+\vartheta_{i+1,i+2}= 2\pi \ \text{ for } 1\le i\le 5
\end{equation}
and thus that $X$ is uniquely determined
% showing that $a_{\bf q}$ has a unique psd completion.
To see why the identities (\ref{eqangle2}) hold, we use the well known fact that the angles $\vartheta_{ij}$ satisfy the (triangle)  inequalities:
\begin{equation}
\label{eqangle3}
\vartheta_{12}+\vartheta_{23}+\vartheta_{13}\le 2\pi,\ -\vartheta_{13}-\vartheta_{14}+\vartheta_{34}\le 0, \ \vartheta_{14}+\vartheta_{45}+\vartheta_{15}\le 2\pi.
\end{equation}
Summing up the three inequalities in (\ref{eqangle3}) and combining with (\ref{eqangle1}), we deduce that equality holds throughout in (\ref{eqangle3}).
This permits to derive the values of $\vartheta_{13}=\pi/2$ and $\vartheta_{14}=\pi/4$ and proceed analogously for the remaining angles. 
(For details  on the parametrization of positive semidefinite matrices using the $\arccos$ map, see \cite{BJT} or \cite{Lau97}).

\section{Universal rigidity of tensegrity frameworks}\label{sec:con}
 Our goal in this section  is to give  a concise and self-contained treatment  of some known results concerning the universal rigidity of tensegrity frameworks. In particular, building on ideas from the two previous sections  we give a very short and elementary proof of Connelly's sufficient condition for universal rigidity for both generic and non-generic tensegrity frameworks. Lastly, we also investigate the relation of our sufficient condition from Theorem~\ref{thm:sphcorp} (for the Gram setting) to Connelly's  sufficient condition from Theorem \ref{thm:con} (for the Euclidean distance setting).

%We recall some definitions. % concerning tensegrity frameworks.

\subsection{Connelly's characterization} 
 %and will be denoted as $G({\bf p})$. 
The framework $G(\p)$ is  called  {\em $d$-dimensional} if $p_1,\cdots,p_n \in \mr^d$ and   their affine span is  $\mr^d$. %We will denote by  $\mathcal{C}^d(G)$  the set of all $d$-dimensional frameworks of $G$.  
A $d$-dimensional framework is said to be in {\em general position} if every  $d+1$ vectors are   affinely independent. 
Given a framework $G(\p)$ in  $ \mr^d$,   its {\em configuration matrix} is  the $n\times d$ matrix $P$ whose rows are the vectors $p_1^\sfT,\ldots,p_n^\sfT$, %correspond to  the vectors  of  the framework, thus 
so that   $PP^\sfT=\gram(p_1,\ldots,p_n)$. % is the Gram matrix of the vectors $p_1,\cdots,p_n$.
The  framework $G(\p)$ is  said to be  {\em generic} if the  coordinates of the vectors $p_1,\ldots,p_n$    are algebraically independent over the rational numbers.

\begin{definition}\label{defdominate}
Let $G=([n],E) $ be a tensegrity graph with $E=B \cup C\cup S$. A  tensegrity  framework $G(\p)$  is said to {\em dominate} a  tensegrity framework $G({\bf q})$ if the following conditions hold:
\begin{itemize}
\item[(i)]$\|p_i-p_j\|=\|q_i-q_j\|$ for all (bars) $ \{i,j\} \in B$,
\item[(ii)]$ \|p_i-p_j\| \ge \|q_i-q_j\|$ for all (cables) $\{i,j\}\in C$,
\item[(iii)]$ \|p_i-p_j\| \le   \|q_i-q_j\|$ for all (struts) $\{i,j\} \in S$.
\end{itemize}
\end{definition}

Two frameworks  $G(\p)$ and $G({\bf q})$  are  called {\em congruent} if 
$$\|p_i-p_j\|=\|q_i-q_j\|, \ \forall i\not=j  \in [n].$$ 
Equivalently, this means that  $G(\q)$ can be obtained by $G(\p)$ by a rigid motion  of the Euclidean space. In this section we will be concerned  with tensegrity frameworks   which,  up to  the group of rigid motions of the Euclidean space, admit a unique realization.
\begin{definition} A tensegrity framework  $G(\p)$ is called   {\em universally rigid} if it is 
congruent to any  tensegrity it dominates. 
 %Equivalently, a framework $G(\p)$ is universally rigid if the corresponding $G$-partial EDM has a {\em unique} completion to a full EDM. 
\end{definition}

%\begin{definition}\label{def:conicinfty} Given a framework $G(\p)$, the vectors $p_i-p_j$ for all $\{i,j\} \in E$ are called its {\em edge directions}.  Following Connelly~\cite{Cb}, we say that the  edge directions of a framework $G(\p)$ in $\oR^d$    {\em  lie on a   conic at infinity} if there exists a nonzero matrix $R \in \mcs^d$ such that 
%\begin{equation}\label{conic}
%(p_i-p_j)^\sfT R\ (p_i-p_j)=0\  \text{ for all } \{i,j\} \in E.
%\end{equation}
%\end{definition}
 
%%%%%%%%%%%%%%%%%%%%start
\if
The following  result gives  a geometric characterization of tensegrity frameworks whose edge directions do not lie on a conic at infinity as  the extreme points of a certain  spectrahedron.
%that fail to satisfy \eqref{conic} as extreme points of a certain spectrahedron. 

\begin{lemma}\label{lem:conicextreme} The edge directions of a tensegrity  framework  $G(\p)$   do not lie on a conic at infinity if and only if  the Gram matrix of the vectors $p_1,\ldots,p_n$   is an extreme point of the spectrahedron  
\begin{equation}\label{eq:spectg}
\left\{X\psd 0 : \la F_{ij},X \ra=\|p_i-p_j\|^2 \  \text{ for } \{i,j\} \in E\right\}.
\end{equation}
\end{lemma}

\begin{proof}Let $P$ be the $n\times d$ matrix with the vectors  $p_1,\ldots,p_n$ as rows so that  $\gram(p_1,\ldots,p_n)=PP^\sfT$. As  $PF_{ij}P^\sfT =(p_i-p_j)(p_i-p_j)^\sfT$ the claim follows by  applying Corollary~\ref{thm:conic}.
%We have that  $\la P^\sfT F_{ij}P\ra =(p_i-p_j)(p_i-p_j)^\sfT, \forall ij \in E$ so  the claim   follows directly  from  Theorem~\ref{thm:conic}. \qed
\qed \end{proof} 
 \fi
 
 \medskip
 An essential ingredient for characterizing   universally rigid tensegrities  is the notion of  equilibrium stress matrix which we now  introduce. 
 
\begin{definition} \label{defstress}
A matrix $\Omega \in \mcs^n$ is called an  {\em equilibrium stress matrix}  for   a tensegrity  framework $G(\p)$ if it satisfies:
\begin{itemize}
\item[(i)] $\Omega_{ij}=0$ for all $\{i,j\} \in {\overline E}$. % (i.e., all pairs $\{i,j\}\not\in E$ with $i\ne j$).
\item[(ii)]  $\Omega e=0$ and $ \Omega P=0$, i.e., $\sum_{j\in V} \Omega_{ij} p_j=0$ for all $i\in V$.
\item[(iii)] $\Omega_{ij}\ge 0$ for all  (cables) $\{i,j\} \in C$ and $\Omega_{ij}\le 0$ for all  (struts) $\{i,j\} \in S$.

\end{itemize}
\end{definition}
Note that, by property (i) combined with the condition $\Omega e=0$, any equilibrium stress matrix  $\Omega$ can be written as $\Omega=\sum_{\{i,j\} \in E} \Omega_{ij} F_{ij}$, where we set  $F_{ij}=(e_i-e_j)(e_i-e_j)^\sfT$.

\smallskip
The following  result (Theorem \ref{thm:con}),  due to R. Connelly, establishes  a {\em sufficient condition} for determining the universal rigidity  of  tensegrities.   All the ingredients  for  its proof are   already present  in~\cite{C82} although there is no explicit statement of the theorem there. An exact  formulation  and a proof of Theorem~\ref{thm:con} can be found in the (unpublished) work ~\cite{Cb}.  We  now give    an elementary proof  of   Theorem~\ref{thm:con} which relies only on basic properties of positive semidefinite matrices. Our proof goes along the same lines as the proof of Theorem \ref{thm:sphcorp} above and it  is substantially shorter and simpler  in comparison  with Connelly's original proof.

\medskip
\begin{theorem}\label{thm:con} 
Let  $G=([n],E)$ be a tensegrity graph with $E=B\cup C\cup S$ and let $G(\p)$ be a tensegrity framework in $\oR^d$ such that   $p_1,\ldots,p_n$ affinely span $\oR^d$. Assume there exists  an equilibrium stress matrix
 $\Omega$   for $G(\p)$ such that:
\begin{itemize}
\item[(i)] $\Omega$ is positive semidefinite.
\item[(ii)] $\Omega$ has corank $d+1$.
\item[(iii)] %The edge directions of $G(\p)$ for which $\Omega_{ij}\not=0$  do not lie on a conic at infinity, i.e., 
 For any matrix $R \in \mcs^d$ the following holds:
 \begin{equation}\label{CC}
  (p_i-p_j)^{\sfT}R\ (p_i-p_j)=0 \ \forall \{i,j\} \in B \cup\{ \{i,j\}\in C\cup S:  \Omega_{ij}\not=0\} \Longrightarrow R=0.
  \end{equation}
\end{itemize}
Then, $G(p)$ is universally rigid.
\end{theorem}

\begin{proof}Assume that   $G(\p)$ dominates another  framework $G(\q)$, 
our goal is to  show that  $G(\p)$ and  $G(\q)$ are congruent. Recall that $P$ is the $n\times d$ matrix with the vectors $p_1,\cdots,p_n$ as rows and define the augmented $n \times (d+1)$ matrix $P_a=\left(\begin{matrix} P & e \end{matrix}\right)$  obtained by adding the all-ones vector as last column to $P$.
Set $X=PP^T$ and 
$X_a=P_aP_a^\sfT$, so that $X_a=X+ee^\sfT$.
As the tensegrity $G(\p)$ is $d$-dimensional, we have that $\rankspace \ X_a=d+1.$ %As $\mcp_G=\ms^n_+\cap \{ Y \in \mcs^n :  \la F_{ij},X \ra=\|p_i-p_j\|^2  \text{ for } ij\in E \} $ it follows that\footnote{the space is affine} $\mF_{\mcp_G}(X)=\{ Y \in \mcp_G : \Ker X \subseteq \Ker Y\}.$ 
We claim that 
$\Ker  X_a=\ran \Omega$.
Indeed, as  $\Omega$ is an equilibrium stress matrix for $G(\p)$, we have that $\Omega P_a=0$ and thus $\Omega X_a=0$.
This implies that $\ran X_a \subseteq \Ker \Omega$ and, as  $\cor  \Omega=d+1=\rankspace  X_a$, it follows that   $\Ker X_a=\ran \Omega$.
%Moreover, $\Omega X=0$ implies $\langle \Omega,X\rangle =0$ since both $\Omega$ and $X$ are psd.

Let   $Y$ denote the Gram matrix of the vectors $q_1,\cdots,q_n$. We claim that $\Ker Y \supseteq \Ker X_a$.
Indeed, we have that
\begin{equation}\label{eq}
0\le \la\Omega,Y\ra=\left\la\sum_{\{i,j\}\in E}\Omega_{ij}F_{ij},Y\right\ra \le  \sum_{\{i,j\}\in E}\Omega_{ij}\la F_{ij},X_a\ra=\la \Omega,X_a\ra=0.
\end{equation}
 The  first inequality follows from the fact that $\Omega,Y\succeq 0$;   the second inequality holds since
   $\Omega_{ij}\la F_{ij},Y\ra \le \Omega_{ij}\la F_{ij},X\ra=\Omega_{ij}\la F_{ij},X_a\ra$ for all edges $\{i,j\}\in E$, using the fact that    $G(\p)$ dominates $G(\q)$ and the sign conditions on $\Omega$.
 %$\Omega_{ij}\ge 0$ for $\{i \in C$ and $\Omega_{ij}\le 0, \forall ij\in S$. 
Therefore equality holds throughout in   \eqref{eq}. This gives  $\la \Omega,Y\ra=0$, implying  $Y\Omega=0$ (since $Y, \Omega\succeq 0$)
and thus    $\Ker Y \supseteq \ran  \Omega=\Ker X_a$. 
  
  As   $\Ker Y \supseteq \Ker X_a$, we deduce that  $\Ker (Y-X_a)\supseteq \Ker X$ and thus $Y-X_a$ can be written as %  there exists a matrix $R\in \mcs^{d+1}$ such that 
\begin{equation}\label{eq2}
Y-X_a=P_aRP_a^\sfT \ \text{ for some matrix } \ R=\left(\begin{matrix}A & b\cr b^\sfT & c\end{matrix}\right)\in \mcs^{d+1},
\end{equation}
 where $A \in \mcs^d$, $b\in \oR^d$ and $c\in \oR$.
 
 As  equality holds throughout in  \eqref{eq} holds,
  we obtain  $\Omega_{ij}\la F_{ij},Y-X_a\ra=0$ for all $\{i,j\} \in C\cup S$. Therefore,
   $\la F_{ij},P_aRP_a^\sfT \ra=(p_i-p_j)^\sfT A(p_i-p_j)=0 $ 
for all  $\{i,j\} \in B$ and for all $\{i,j\}\in C\cup S$ with $\Omega_{ij}\not=0$. Using condition (iii), this   implies that $A=0$. 
Now, using  \eqref{eq2} and the fact that $A=0$, we obtain that   $$(Y-X_a)_{ij}=b^\sfT p_i+b^\sfT p_j+c\ \ \text{ for all } i,j \in [n].$$ From this follows that 
 $$\|q_i-q_j\|^2=Y_{ii}+Y_{jj}-2Y_{ij}=(X_a)_{ii}+(X_a)_{jj}-2(X_a)_{ij}= \|p_i-p_j\|^2$$ for all $ i,j\in [n],$ thus showing that $G({\bf p})$ and $G({\bf q})$ are congruent. \qed
\end{proof}

\begin{comment}
The first step in the proof is to show   that $X\in {\rm relint} \mcp_G$, i.e.,   $\mF_{\mcp_G}(X)=\mcp_G$.  For this, it is enough to show the inclusion   $\mcp_G \subseteq \mF_{\mcp_G}(X)$.

 Let $Y \in \mcp_G$  so by definition $\la F_{ij}, X-Y\ra =0, \forall ij \in E$. Then $\la \Omega, X-Y\ra=\la \sum_{ij \in E}\Omega_{ij}F_{ij}, X-Y\ra=0$ and since $\Omega P_a=0$ it follows that $\la \Omega, Y\ra=0$.  As $\Omega, Y\psd 0$  we get that  $\Omega Y=0$ and thus $\Ker Y\supseteq \ran \Omega=\Ker X.$

Next we investigate the structure  of elements in $\per_{\mcp_G}(X)$. Applying  \eqref{eq:pert} we get  that $\per_{\mcp_G}(X)=\{P_aRP_a^\sfT: \la  F_{ij},P_aRP_a^\sfT\ra=0, \forall ij \in E,\  R\in \mcs^{d+1} \}.$
Write  $R=\left[\begin{array}{cc}
A& b\\
b^\sfT& c
\end{array}\right]$ and as $G(\p)$ does not lie on a conic at infinity we get that  $A=0$. 

We are now ready to conclude the proof. Consider a framework $G({\bf q})$ equivalent to $G(\p)$ and let $Y=\gram(q_1,\ldots,q_n)$. Clearly, $Y\in \mcp_G=\mF_{\mcp_G}(X)$  which implies that $Y-X \in \per_{\mcp_G}(X)$. This implies that  there exist $b\in \mr^d$ and $c \in \mr$ such that  $(Y-X)_{ij}=b^\sfTp_i+b^\sfTp_j+c, \forall i,j \in [n]$. Then $\|q_i-q_j\|^2=Y_{ii}+Y_{jj}-2Y_{ij}=X_{ii}+X_{jj}-2X_{ij=}\|p_i-p_j\|^2, \forall i,j\in [n].$
\end{comment}
 \medskip

Notice that the assumptions of the theorem imply that $n\ge d+1$. Moreover, for $n=d+1$ we get that $\Omega$ is the zero matrix in which  case~\eqref{CC} is satisfied only for $G=K_n$ and $C=S=\emptyset$. Hence Theorem \ref{thm:con} is useful only in the case when $n\ge d+2$. 
\medskip 

There is a natural pair of primal and dual semidefinite programs attached to a given tensegrity framework $G({\bf p})$:
 \begin{equation}\label{primalGp}
 \begin{array}{ll}
% \underset{X}
 {\sup}_X\ \{0: X\succeq 0, & \la F_{ij},X \ra=\|p_i-p_j\|^2 \  \text{ for } \{i,j\} \in B,\\
& \la F_{ij},X \ra \le \|p_i-p_j\|^2 \  \text{ for } \{i,j\} \in C,\\
& \la F_{ij},X \ra \ge \|p_i-p_j\|^2\   \text{ for } \{i,j\} \in S\},
\end{array}
\end{equation}
\begin{equation}\label{dualGp}
\begin{array}{ll}
%\underset{z,Z}
{\inf}_{y,Z}\ \{ \sum_{ij \in E}y_{ij}\|p_i-p_j\|^2: & Z=\sum_{ij \in E}y_{ij}F_{ij}\succeq 0,\\
&  y_{ij}\ge 0 \ \text{ for }  \{i,j\}  \in C,\\
&   y_{ij}\le 0\  \text{ for } \{i,j\} \in S\}.
\end{array}
\end{equation}
 The  feasible
  (optimal) solutions of the primal program (\ref{primalGp}) correspond to the frameworks $G({\bf q})$ that are dominated by  $G({\bf p})$, while
   the optimal solutions to the dual program (\ref{dualGp}) correspond to the positive semidefinite  equilibrium stress matrices for the tensegrity 
framework   $G({\bf p})$.  %Moreover, having a   psd equilibrium stress matrix satisfying the assumptions (i) and (ii) of Theorem~\ref{thm:con}   corresponds to having  a strict complementary pair of primal and dual optimum solutions.  

Both matrices $X=PP^\sfT$ and $X_a=P_aP_a^\sfT$ (defined in the proof of Theorem \ref{thm:con}) are  primal optimal, with $\rankspace X=d$ and $\rankspace X_a=d+1$. Hence,  a   psd equilibrium stress matrix $\Omega$ satisfies the conditions (i) and (ii) of Theorem~\ref{thm:con}   precisely when the pair $(X_a,\Omega)$ is a strict complementary pair  of primal and dual optimal solutions.  

In the case of bar frameworks (i.e., $C=S=\emptyset$), the condition  (iii)  of Theorem \ref{thm:con}   
expresses the fact that the matrix $X=\gram(p_1,\ldots,p_n)$    is an extreme point of the feasible region of (\ref{primalGp}).
Moreover, 
 $X_a$ lies in its relative interior (since  $\Ker Y\supseteq \Ker X_a$ for any primal feasible $Y$, as shown in the above proof of Theorem \ref{thm:con})).

\begin{remark}
In the terminology of Connelly, the condition (\ref{CC}) says that {\em the edge directions} $p_i-p_j$ of $G(\mathbf p)$ for all edges $\{i,j\}\in B$ and all edges $\{i,j\}\in C\cup S$ with nonzero stress $\Omega_{ij}\ne 0$ {\em do not lie on a conic at infinity}.

Observe that this condition   %(that the edge directions of $G(\p)$ corresponding to a nonzero stress   do not lie on a conic at infinity) 
cannot be omitted in Theorem \ref{thm:con}.
This is illustrated by  the following example, taken from~\cite{A10b}.  Consider the graph $G$ on 4 nodes with edges
$\{1,2\}$, $\{1,3\}$, $\{2,3\}$ and $\{2,4\}$,
% set $E(G)=\{ \{1,2\},\{1,3\},\{2,3\},\{2,4\}\}$
 and  the $2$-dimensional bar framework $G(\p)$ given by  
 $$p_1=(-1,0)^\sfT,\  p_2=(0,0)^\sfT,\ p_3=(1,0)^\sfT  \text{ and } p_4=(0,1)^\sfT.$$% illustrated in Figure~\ref{k13}. 
Clearly, the framework $G(\p)$ is not  universally rigid (as one can rotate $p_4$ and get a new framework, which is equivalent but not congruent to $G({\bf p})$). 
 On the other hand, the  matrix $\Omega=(1,-2,1,0)(1,-2,1,0)^\sfT$ is the only  equilibrium stress matrix for $G(\p)$, it is positive semidefinite  with corank $3$. Observe however that the condition (\ref{CC}) does not hold (since the nonzero matrix $R=e_1e_2^\sfT +e_2e_1^\sfT$ satisfies $(p_i-p_j)^\sfT R(p_i-p_j)=0$ for all $\{i,j\}\in E$).
 %However, the  edge directions of $G(\p)$ that correspond to a nonzero stress are $p_1-p_2, p_1-p_3, p_2-p_3$ which  are all parallel to $(1,0)^\sfT$ and thus   lie on a conic at infinity. 
\end{remark}

%\begin{figure}[h]
%\centering \includegraphics[scale=0.5]{k13}
%\label{k13}
%\caption{A 2-dimensional framework, not universally rigid}
% showing that  assumption $(iii)$ in Theorem~\ref{thm:con} is necessary. }
%\end{figure}

\subsection{Generic universally rigid frameworks}\label{sec:generic}

It is natural to ask for a converse  of Theorem~\ref{thm:con}. This question  has been settled  recently in \cite{GT} in  the affirmative  for generic frameworks (cf.~Theorem~\ref{thmGT}).
First, we  show that, for generic frameworks, the `no conic at infinity' condition (\ref{CC}) can be omitted since it holds automatically. This result was obtained  in \cite{Con} (Proposition 4.3), but for the sake of completeness we have included a different   and more explicit argument. 

We need  some  notation. Given  a framework $G({\bf p})$ in $\oR^k$,   we let $\mcp_{\p}$ denote  the  $\binom{k+1}{2} \times |E|$ matrix,  whose $ij$-th column contains the entries of the upper triangular part of the  matrix 
$(p_i-p_j)(p_i-p_j)^\sfT\in \mcs^k$. For  a subset $I\subseteq E$,  $\mcp_{\p}(I)$ denotes  the $\binom{k+1}{2}\times |I|$ submatrix of $\mcp_{\p}$ whose columns are indexed by edges in $I$. %Notice that for this definition to make sense it should the case that   $|E|\ge \binom{k+1}{2}$.

\begin{lemma}\label{lem:pol} Let $k\in \oN$ and let $G=([n],E)$ be a graph on  $n \ge k+1$ nodes and with minimum degree at least $k$. Define the polynomial $\pi_{k,G}$ in $kn$ variables by 
%$$\pi_{k,G}: \mr^{nk} \mapsto \mr \text{ where }
$$ \pi_{k,G}(\p)=\sum_{I\subseteq E, |I|=\binom{k+1}{2}} (\det \mcp_{\p}(I))^2$$ 
for ${\bf p}=\{p_1,\ldots,p_n\}\subseteq (\oR^k)^n$.
Then, the polynomial  $\pi_{k,G}$ has integer coefficients and it is not identically zero.
\end{lemma}

\begin{proof}Notice that for   the specific choice of parameters we have that   $|E|
 \ge {nk \over 2}\ge {(k+1)k \over 2}$.  It is clear that $\pi_{k,G}$ has integer coefficients.
We show  by induction on $k\ge 2$ that for every graph $G=([n],E)$ with  $n\ge k+1$ nodes and minimum degree at least $k$ the polynomial $\pi_{k,G}$ is not identically zero. 
 
 For  $k=2$, we  distinguish two cases: $(i)\  n=3$  and $(ii)\  n \ge 4$. In case (i),  $G=K_3$ and, for the vectors $p_1=(0,0)^\sfT, p_2=(1,0)^\sfT, p_3=(0,1)^\sfT$, we have that $\pi_{2,G}(\p)\not=0$. In case (ii), we can now assume without loss of generality  that the edge set contains the following subset  $I=\{\{1,2\},\{1,3\}, \{2,4\}\}$. 
 For the vectors  $p_1=(0,0)^\sfT, p_2=(1,0)^\sfT, p_3=(0,1)^\sfT, p_4=(2,1)^\sfT$, %For $I=\left\{\{12\},\{13\},\{24\}\right\}$ 
  we have that $\det \mcp_{\p}(I)\not=0$ and thus $\pi_{2,G}(\p)\not=0$.
 
 Let $k\ge 3$ and  consider a graph $G=([n],E)$ with $n\ge k+1$ and minimum degree at least $k$. Let   $G\setminus n$ be the graph obtained from  $G$ by  removing  node $n$ and all edges  adjacent to it. Then, $G\setminus n$ has at least $k$ nodes and minimum degree at least $k-1$.  Hence, by the induction hypothesis, the polynomial $\pi_{k-1,G\setminus n}$ is not identically zero. Let   ${\bf p}=\{p_1,\ldots, p_{n-1}\}\subseteq  \mr^{k-1}$ be a generic set of vectors and define ${\bf \tilde p}=\{\tilde p_1,\ldots,\tilde p_n\}\subseteq \oR^k$, where  $\tilde p_i=(p_i^\sfT,0)^\sfT \in \mr^k$ for $1  \le i \le n-1$   and $\tilde p_n=({\bf 0},1)^\sfT \in \mr^k$. 
 As $\bf p$ is generic,      $\pi_{k-1,G\setminus n}({\bf p})\not=0$  and thus   $\det \mcp_{\p}(I)\not=0$ for some subset
$I \subseteq E(G\setminus n)$ with $|I|=\binom{k}{2}$. % satisfying $\det \mcp_{\q}(I)\not=0$. 
% Assume w.l.o.g. that 
% .  that $N(n) \supseteq \{1,2,\ldots,k\} $ 
 %Define ${\bf \tilde q}=\{\tilde q_1,\ldots,\tilde q_n\}\subseteq \oR^k$, where  $\qt_i=(q_i^\sfT,0)^\sfT \in \mr^k$ for $1  \le i \le n-1$   and $\qt_n=({\bf 0},1)^\sfT \in \mr^k$.
Say, node $n$ is adjacent to the nodes $1,\ldots,k$ in $G$ and define the edge subset $\tilde{I}=I\cup \left\{ \{n,1\},\ldots,\{n,k\}\right\}\subseteq E$. Then,   the matrix   $\mcp_{{\bf \tilde p}}(\tilde{I})$
has the block-form
\begin{equation*}
  \begin{array}{r@{\,}l}
    & 
    \begin{matrix}
      \mspace{15mu}\overbrace{\rule{2cm}{0pt}}^{\binom{k}{2}} &     
      \overbrace{\rule{2cm}{0pt}}^{k}
    \end{matrix}
    \\
    \mcp_{{\bf \tilde p}}(\tilde{I})=  & 
    \left(\begin{matrix}
   %   & & & & & \\
%      & \mcp_{\q}(I) & & \\
      & \mcp_{\p}(I) &&* &\ldots &* \\[3pt]
      {\bf 0} & \ldots & {\bf 0} & -p_1& \ldots & -p_k\\
      0 & \ldots & 0  &  1& \ldots &  1
    \end{matrix}\right).
  \end{array}
\end{equation*}
As the vectors $p_1,\ldots,p_{n-1}\in \mr^{k-1}$ were chosen to be  generic, every $k$ of them are affinely independent. This implies that the vectors 
$(-p_1^\sfT,1)^\sfT,\ldots(-p_k^\sfT,1)^\sfT$  are linearly independent. Hence,   $\det \mcp_{{\bf \tilde p}}(\tilde{I})\not=0$ and thus $\pi_{k,G}({\bf \tilde p})\not=0$. \qed
\end{proof}

\begin{theorem}\label{thm:gencon}\cite{Con}
Let $G(\p)$ be a  generic  $d$-dimensional  framework and assume that $G$ has minimum degree at least $d$.
 Then the edge directions of  $G(\p)$ do not lie on a conic at infinity; that is, the system $\{(p_i-p_j)(p_i-p_j)^\sfT: \{i,j\}\in E\}\subseteq \mcs^d$ has full rank $d+1\choose 2$. 
\end{theorem}

\begin{proof}As the framework $G({\bf p})$ is $d$-dimensional, $G$ must have at least $ d+1$ nodes. By Lemma~\ref{lem:pol}, the  polynomial $\pi_{d,G}$ is not identically zero and thus, since $G(\p)$ is generic, we have that $\pi_{d,G}(\p)\not=0$. By definition of $\pi_{d,G} $ there exists $I\subseteq E$ with $|I|=\binom{d+1}{2}$  such that  $\det \mcp_{\p}(I)\not=0$. This implies that the system 
$\{(p_i-p_j)(p_i-p_j)^\sfT: \{i,j\}\in E\}\subseteq \mcs^d$ has full rank $d+1\choose 2$. %, which  shows the claim.
\qed 
\end{proof}
\medskip

Next we  show that for {\em generic}  frameworks Theorem~\ref{thm:con} remains valid even when~\eqref{CC} is omitted.

\begin{corollary}\label{cor:generic} \cite{Con}
Let $G(\p)$ be a generic $d$-dimensional tensegrity  framework. %where \textcolor{red}{$d\le n-2$}. 
Assume that there exists a positive semidefinite equilibrium stress matrix $\Omega$ with  corank $d+1$. Then $G(p)$ is universally rigid. 
\end{corollary}

\begin{proof}Set $E_0=\{\{i,j\}\in E: \Omega_{ij}\ne 0\}$ and 
define the subgraph $G_0=([n],E_0)$ of $G$. 
First we show that $G_0$ has minimum degree at least $d$.
For this, we use the equilibrium conditions: For all $i\in [n]$, 
$\sum_{j: \{i,j\}\in E_0} \Omega_{ij} p_j=0$, which give an affine dependency among the vectors $p_i$ and $p_j$ for $\{i,j\}\in E_0$.
 By assumption,  $\mathbf p$ is generic and thus in general position, which implies that any $d+1$ of the vectors $p_1,\ldots,p_n$ are affinely dependent. From this we deduce that each node $i\in [n]$ has degree at least $d$ in $G_0$.
 
 Hence we can apply Theorem \ref{thm:gencon} to the generic framework $G_0(\mathbf p)$ and conclude that the system
 $\{(p_i-p_j)(p_i-p_j)^\sfT: \{i,j\}\in E_0\}$ has full rank $d+1\choose 2$. This shows that the condition (\ref{CC}) holds. Now we can apply Theorem \ref{thm:con} to $G(\mathbf p)$ and conclude that $G(\mathbf p)$ is universally rigid.
 % As $G_0({\bf p})$ is a generic $d$-dimensional framework it suffices to show that $G_0$ has minimum degree at least $d$. Indeed, then we deduce from Theorem \ref{thm:gencon} that the edge directions of $G_0({\bf p})$ do not lie on a conic at infinity (which is exactly condition (iii) of Theorem \ref{thm:con}). Thus we can apply Theorem \ref{thm:con} and conclude that $G({\bf p})$ is universally rigid.
%As $\Omega $ is a stress matrix for $G({\bf p})$,  $\Omega_{ii}p_i+\sum_{j: \{i,j\}\in E_0}\Omega_{ij}p_j=0$ holds for any $i\in [n]$. Thus the vectors $\{p_i, p_j \ (\{i,j\}\in E_0)\}$ are affinely dependent. This implies that the degree of node $i$ in $G_0$ is at least $d$ since, as $\bf p$ is generic, it is in general position and thus any $d+1$ vectors are affinely independent.
\qed
%It suffices to show that Without loss of generality  we can assume that $\Omega_{ij}\not =0, \forall \{ij\} \in E$. By assumption, the $d$-dimensional framework $G(\p)$ is generic and thus it also is in general position. Since $\Omega$ is  an equilibrium stress, the equilibrium condition at every node $i \in [n]$ implies  that the vectors $\{p_i\}\cup \{p_j : j\in N(i)\}$ are affinely dependent. In turn this implies that $\deg i \ge  d+1, \forall i \in [n]$. Then Theorem~\ref{thm:gencon} implies that the edge directions of $G(\p)$ do not lie on a conic at infinity. Lastly, the claim follows by Theorem~\ref{thm:con}.\qed
\end{proof}
\medskip

We note that for bar frameworks this fact has been also   obtained  independently  by A. Alfakih using the related concepts of {\em dimensional rigidity}  and {\em Gale matrices}.  The notion of dimensional rigidity was introduced in~\cite{A07} where a sufficient condition was obtained  for showing that a  framework  is dimensionally rigid.  In~\cite{A10},  using the concept of a Gale matrix, this condition was shown to be equivalent to the sufficient condition from Theorem~\ref{thm:con} (for bar frameworks). Lastly, in~\cite{A10} it is shown that for generic frameworks the notions of dimensional rigidity and universal rigidity coincide.

%Example 3.1 from~\cite{A07} ruins  any  hope for a converse of Theorem~\ref{thm:con} when $G(\p)$ is {\em not generic}. On the other hand for generic frameworks   the converse of Theorem~\ref{thm:con}  was proved only recently. 

In the special case of bar frameworks, the converse of Corollary~\ref{cor:generic} was proved recently by  S.J. Gortler and P. Thurston. 

\begin{theorem}\cite{GT} \label{thmGT}
Let $G(\p)$ be a  generic $d$-dimensional bar framework and assume that it is   universally rigid. Then  there exists a positive semidefinite equilibrium stress matrix $\Omega$ for $G(\p)$ with corank $d+1$.
\end{theorem} 

\subsection{Connections with unique completability}\label{sec:connections}

 In this section we investigate   the  links  between the two notions of universally completable   and universally rigid tensegrity frameworks.  We start the discussion with defining the suspension of a tensegrity framework.

\begin{definition}\label{defsuspension}
Let   $G=(V=[n],E)$ be a tensegrity graph with $E=B\cup C \cup S$.  We denote by  $\nabla G=(V\cup \{0\},E')$ its {\em suspension tensegrity  graph}, with $E'=B'\cup C'\cup S'$ where $B'=B\cup\{\{0,i\} : i\in [n]\}$, $C'=S$ and $S'=C$.
%obtained by adding a new node to $G$,  denoted by 0, and %making it adjacent to all nodes of $G$.
Given a  tensegrity framework $G(\p)$,  
%${\bf p}=\{p_1,\ldots,p_n\}$ attached to the nodes of $V$,
we define the {\em extended tensegrity framework} 
$\nabla G(\widehat {\bf p})$ where $\widehat{p}_i=p_i$ for all $i \in [n]$ and $\widehat{p}_0={\bf 0}$.
%= \{p_0=0, p_1,\ldots,p_n\}$ attached to the nodes of $\nabla G$, whose apex node is thus located in the origin.
\end{definition}

Our first observation is a correspondence between the universal completability  of a tensegrity framework  $G({\bf p})$ and 
the universal rigidity of its extended tensegrity framework  $\nabla G(\widehat{\bf p})$. The analogous  observation in the setting of global rigidity  has been also made  in~\cite{CW} and~\cite{SC}. 

\begin{lemma}\label{lem:connection}
 Let   %$G=(V=[n],E)$ be a tensegrity graph with $E=B\cup C\cup S$ and 
 $G(\bf p)$ be a tensegrity framework and let $\nabla G(\widehat \bf p)$ be its extended tensegrity framework as defined in Definition \ref{defsuspension}.
Then, the  tensegrity framework $G({\bf p})$ is universally completable if and only if the extended tensegrity  framework $\nabla G(\widehat{\bf p})$ is universally rigid.
\end{lemma}

\begin{proof}Notice that for any family of vectors  $q_1,\ldots,q_n$, their Gram matrix  satisfies the conditions:
\begin{align*}
\la E_{ij},X\ra &= p_i^\sfT p_j \ \text{ for all } \{i,j\}\in V\cup B,\\
\la E_{ij},X\ra &\le  p_i^\sfT p_j \ \text{ for all } \{i,j\} \in C,\\
\la E_{ij},X\ra &\ge p_i^\sfT p_j \ \text{ for all } \{i,j\} \in S,
\end{align*}
if and only if the Gram matrix  of $q_0={\bf 0}, q_1,\ldots,q_n$ satisfies:
\begin{align*}
\la F_{ij},X \ra &=\|p_i-p_j\|^2 \ \text{ for all } \{i,j\}\in B',\\
\la F_{ij},X \ra &\le\|p_i-p_j\|^2 \ \text{ for all } \{i,j\}\in C',\\
\la F_{ij},X \ra &\ge\|p_i-p_j\|^2 \ \text{ for all } \{i,j\}\in S',
\end{align*}
which implies the claim. 
\end{proof}\qed
\medskip
%As we saw above, Theorem \ref{thm:con} provides sufficient conditions for the universal rigidity and for the unique psd completability of frameworks. We now observe that these two sets of sufficient conditions are in fact equivalent when applied, respectively, to the frameworks $\nabla G(\widehat{\bf p})$ and
%to $G({\bf p})$.

In view of Lemma~\ref{lem:connection} it is reasonable to ask whether Theorem~\ref{thm:sphcorp} can be derived from Theorem~\ref{thm:con} applied to the tensegrity framework $\nabla G(\widehat{\p})$. We will show that this  is the case for bar frameworks, i.e., when $C=S=\emptyset$. Indeed, for  a bar framework,  the condition (\ref{CC})  from Theorem~\ref{thm:con} applied to the suspension tensegrity  framework $\nabla G(\widehat{\p})$ becomes
 $$ R \in \mcs^d, \ (p_i-p_j)^\sfT R\  (p_i-p_j)=0 \text{ for all } \{i,j\} \in E \cup \{\{0,i\}: i\in [n]\} \Longrightarrow R=0,$$
and, as   $\widehat{p}_0=0$, this   coincides  with the condition \eqref{eq:sphconic}.

\smallskip
The following  lemma shows  that   for bar frameworks there exists  a one to one correspondence  between equilibrium stress matrices for $\nabla G(\widehat{\p})$ and spherical stress matrices for $G(\p)$. The crucial fact that we use here is   that for bar frameworks there are no sign conditions for a spherical  stress matrix for $G(\p)$  or for an equilibrium stress matrix for $\nabla G(\widehat{\p})$.

\begin{lemma}\label{lem:conststress}Let $G({\bf p})$ be a bar framework in $\oR^d$ such that $p_1,\ldots,p_n$ span linearly $\oR^d$.
%Let $\bf p$ be a $(d+1)$-dimensional framework in $\oR^{d+1}$.
The following assertions are equivalent:
\begin{itemize}
\item[(i)] There exists an equilibrium stress matrix $\Omega\in \mcs^{n+1}_+$ for the framework $\nabla G(\widehat{\bf p})$ with $\cor \Omega=d+1$.
% and satisfying 
%$\Omega_{ij}=0$ for all pairs $\{i,j\}$ of distinct nodes that are not adjacent in $\nabla G$.
\item[(ii)]
There exists a spherical stress matrix for $G(\p)$. %satisfying the assumptions  of Theorem~\ref{thm:sphcorp}.
\end{itemize}
\end{lemma}

\begin{proof}Let $P\in \oR^{n\times d}$ be the configuration matrix of the framework $G(\p)$ 
and let 
%$P_a= \left(\begin{matrix} P & e \end{matrix}\right)$ 
 $\widehat P_a =\left(\begin{matrix} {\bf 0} & 1 \cr P & e \end{matrix}\right)$.
Write a matrix $\Omega\in \mcs^{n+1}_+$ in block-form as
\begin{equation}\label{eqOmega}
\Omega =\left(\begin{matrix} w_{0} & w^\sfT\cr w & Z\end{matrix}\right) \ \text{ where } Z\in \mcs^n_+, w\in \oR^n, w_{0}\in \oR.
\end{equation}
Notice that $\Omega$ is supported by $\nabla G$ precisely when $Z$ is supported by~$G$. The matrix $\Omega$ is a stress matrix for $\nabla G(\widehat{\bf p})$ if and only if $\Omega \widehat P_a=0$ which is equivalent to
\begin{equation}\label{eqsystem}
ZP=0,\ w = -Ze,\ w_{0}=-w^\sfT e.
\end{equation}
Moreover, $\Ker \Omega = \ran \widehat {P_a}$ if and only if $\Ker Z= \ran P$, so that $\corank \Omega =d+1$ if and only if $\corank Z=d$.
The lemma now follows easily: If $\Omega$ satisfies (i), then its principal submatrix $Z$ satisfies (ii).
Conversely, if $Z$ satisfies (ii), then the matrix $\Omega$ defined by (\ref{eqOmega}) and (\ref{eqsystem}) satisfies (i).
\qed\end{proof}
\medskip

Summarizing,   we have established that in the special case of a bar framework $G(\p)$ (i.e., $C=S=\emptyset$),   Theorem~\ref{thm:sphcorp} is equivalent to Theorem~\ref{thm:con} applied to the extended bar framework $\nabla G(\widehat{\p})$. It is not clear % remains open to verify 
whether this equivalence remains valid  for arbitrary tensegrity frameworks. To deal with such frameworks,   Lemma~\ref{lem:conststress} has to be generalized so as to  accommodate the sign conditions for the sperical stress matrix and the equilibrium stress matrix for $G(\p)$ and $\nabla G(\widehat{\p})$, respectively.

\section{The Strong Arnold Property and graph parameters}\label{sec:parameters}

In this section we   revisit  the Strong Arnold Property (SAP) and we show that matrices fulfilling the SAP posses some nice geometric properties. We also  show that psd matrices fulfilling the SAP can be characterized as nondegenerate solutions of some appropriate semidefinite program. Additionally,  we  investigate  the  relation between the graph parameters $\nu^=(\cdot)$ and $\gd(\cdot)$, introduced in \cite{H96,H03} and \cite{LV,LV12}, respectively. 
%Our main result  is a reformulation of the graph parameter $\nu^=(\cdot)$ in terms of $\gd(\cdot)$. 

\subsection{The Strong Arnold Property}\label{sec:sap} 

For a graph   $G=(V=[n],E)$   consider the linear space 
$$\mcc(G)=\{ X\in \mcs^n : \la E_{ij},X\ra=0\  \ \forall \{i,j\} \in \E\}.$$
%=\la E_{ij}\mid ij \in E\cup V\ra. %=\cap_{ij  \in \E} \la E_{ij}\ra^\sfT.

%For two   linear spaces $U,V$ let  $U+V$  denote the subspace spanned by their union. Clearly,  $U^{\perp}\cap V^{\perp}=(U+V)^{\perp}$,  which implies $$\mathcal{C}(G)^{\perp}=\la E_{ij} \mid ij \in \E\ra.$$

\begin{definition}For a graph $G=([n],E)$,  a matrix    $M\in \mcc(G)$ is said to satisfy  the {\em Strong Arnold Property  (SAP) } if 
\begin{equation}\label{def:sap} 
\mathcal{T}_M+\lin\{ E_{ij}: \{i,j\} \in V\cup E\}=\mcs^n.
\end{equation}
\end{definition}

The SAP  has received a significant  amount  of attention due to its connection to the Colin de Verdi\`ere graph parameter $\mu(\cdot)$, introduced and studied in \cite{CdV90}.
The {\em Colin de Verdi\`ere number}  $\mu(G)$ of a graph $G$ is defined as the maximum corank of a matrix $M \in \mcc(G)$ satisfying: 
 $\la E_{ij},M\ra<0$ for all $\{i,j\} \in E$, $M$ has exactly one negative eigenvalue,  and $M$ satisfies the SAP.  The graph parameter $\mu(\cdot)$  
 is   minor monotone, and it turns out that the SAP plays a crucial role for showing this. The importance of the graph  parameter $\mu(G)$ stems in particular   from the fact that it permits to characterize several topological properties  of graphs. For instance, it is known that $\mu(G)\le 3$ if and only if  $ G$ is planar \cite{CdV90}  and $\mu(G)\le 4$ if and only of  $G$ is linklessly embeddable \cite{LS98} (more details can be found  e.g. in \cite{HLS99} and further references therein). 

By  taking orthogonal complements in~\eqref{def:sap}  and using \eqref{eq:normal2},  %the  SAP  is equivalent to  
%$\mathcal{N}_M(\mcr_k)\cap \lin \{E_{ij} :\ \{i,j\}  \in \E\}=0.$ This last fact, combined with ~\eqref{eq:normal} $(ii)$ imply that 
we arrive at  the following  equivalent expression for the  SAP, that we will use in the sequel:
 \begin{equation}\label{eq:sap2}
 X\in \mcs^n,\  MX=0,\  X_{ij}=0 \text{ for all }  \{i,j\}  \in V\cup E \Longrightarrow  X=0.
 \end{equation}

 %%%%%%%%%%%%%%%%%%%%%%%%%%%
\if 0 
\begin{lemma}Let $G=([n],E)$ and consider the semidefinite program $SDP(G):$

\begin{equation*}
\begin{aligned}
& \text{minimize} & & 0 \\
& \text{subject to} & & \la E_{ij},\Omega \ra=0  \text{ for } ij\not \in E\\
& && \Omega \psd 0
\end{aligned}
\end{equation*}

Then $\Omega$ has the SAP for $G$ iff $\Omega$ is a nondegenerate feasible solution for $SDP(G)$.

\end{lemma}
\fi 

Our next goal is to give a geometric characterization of matrices satisfying the SAP using the notion of null space representations. Consider a matrix $M\in \mcs^{n}$,   fix an arbitrary basis for $\Ker M$ and form the  $  n \times \cor M$ matrix that has as  columns the basis elements. The vectors corresponding to the rows of the resulting matrix form a  {\em nullspace representation} of $M$. If we impose structure on $M$ in terms of some graph $G$, nullspace  representations of $M$  exhibit  intriguing geometric  properties and have been extensively studied (see e.g.~\cite{LS99}).  

The next theorem shows that null space representations of matrices satisfying the SAP enjoy  some nice geometric properties. The equivalence between and  the first  and the third item item  has been rediscovered independently by~\cite{vdH97} (Theorem 4.2) and \cite{G} (Lemma~3.1).

\begin{theorem}\label{thm:godsil}
Consider a graph $G=([n],E)$ and a matrix    $M \in \mcc(G)$ with $\cor M=d$.
Let $P\in \mr^{n\times d}$  be a  matrix whose columns form an orthonormal  basis for  $\Ker M$ and  let $\{p_1,\ldots,p_n\}$ denote the row vectors of $P$.  The following  assertions are equivalent: 
\begin{itemize}
\item[(i)] $M$ satisfies the Strong Arnold Property.
\item[(ii)] $PP^\sfT$ is an extreme point of the spectrahedron 
$$ \{X \psd 0: \la E_{ij},X\ra=p_i^\sfT p_j \text{ \rm  for } \{i,j\} \in V\cup E \}.
 $$
\item[(iii)] For any  matrix $R\in \mcs^d$ the following holds:
$$ p_i^\sfT Rp_j=0 \text{ \rm  for all } \{i,j\} \in V\cup E \Longrightarrow  R=0.$$
\end{itemize}
\end{theorem}

\begin{proof}The equivalence $(ii) \Longleftrightarrow(iii) $ follows directly from  Corollary~\ref{thm:conic}.
\smallskip 

\noindent $(i) \Longrightarrow (iii)$ Let $R\in \mcs^d$ such that $p_i^\sfT Rp_j=0$, i.e., $\la PRP^\sfT, E_{ij}\ra =0$ for all  $ \{i,j\} \in V \cup E.$   Thus the matrix  $Y=PRP^\sfT$ belongs to
  $\lin\{ E_{ij} : \{i,j\} \in V\cup E \}^\perp$ and satisfies $MY=0$. By~\eqref{eq:normal2} we have that  $Y\in \mct_M^\perp$ and then   $(i)$ implies $Y=0$ and thus $R=0$ (since $P^\sfT P=I_r$).
  
  % contradicting the hypothesis. 
\smallskip 
\noindent $ (iii) \Longrightarrow (i)$ Write $M=Q
\left(\begin{matrix}\Lambda_1 & 0\\
0 & 0\end{matrix}\right)Q^\sfT$, where $Q=(Q_1\ P)$ is orthogonal and the columns of $Q_1$ form a basis of the range of $M$.
Consider a matrix    $Y \in \mct_M^{\perp}\cap \lin \{E_{ij} : \{i,j\} \in \E\}$. Then, 
 by \eqref{eq:normal},  $Y=PRP^\sfT$ for some matrix  $R \in \mcs^d$. Moreover,  $\la Y,E_{ij}\ra=\la PRP^\sfT,E_{ij}\ra=0$ for all  $\{i,j\} \in V\cup E$, which  by $(iii)$ implies that $R=0$ and thus $Y=0$. 
  \qed
\end{proof}

%$(ii) \Rightarrow (i)$ By \eqref{eq:sap2} we know that  $M$ will fail to satisfy SAP iff there exists some nonzero matrix $X \in \ms^n$ such that $MX=0$ and $X \in \mcc(G)^{\perp}$.  Then, the first condition implies that $X=U\Lambda U^T$ for some nonzero matrix $\Lambda \in \ms^n$ and the second condition that $p_i^T\Lambda p_j=0, \ \forall ij \in E\cup V$. 

%$(i) \Rightarrow (ii)$ Assume  that there exists a nonzero matrix $ Q\in \ms^n$ such that $ p_i^TQp_j=0, \forall ij \in E\cup V.$ Consider the matrix  $X=UQU^T$. Then $MX=0$ (since the columns of $U$ are basis elements for $\Ker M$) and $X \in \mcc(G)$ which implies that $M$ does not satisfy the SAP.\qed

%\begin{remark} Notice that   condition $(ii)$ from  Theorem~\ref{thm:godsil} is equivalent to conditions $(i)$ and $(ii)$ from Lemma~\ref{lem:conic}.  In particular, this gives a new geometric interpretation of the Strong Arnold Property.
%\end{remark}
\medskip 

Our final observation in this section is that a psd matrix having the SAP can 
be also understood as a nondegenerate solution of some semidefinite program. 

\begin{theorem}\label{thm:SAPtosdp}Consider a graph $G=([n],E)$ and let  $M\in  \mcc(G) \cap \mcs^+_n$. The following assertions  are equivalent:
\begin{itemize}
\item[(i)] M satisfies  the Strong Arnold Property.
\item[(ii)] $M$ is a primal nondegenerate solution for the semidefinite program:
\begin{equation*}\label{eq:sap}
\sup_X \{ \la C,X\ra :  \la E_{ij},X \ra=0 \text{ \rm  for }   \{i,j\} \in \E, \ X \psd 0\},
\end{equation*} 
for any $C\in \mcs^n$.
\item[(iii)] $M$ is a dual nondegenerate solution for the dual of the semidefinite program:
\begin{equation}\label{eq:duala}
\sup_X \{0:  \la E_{ij},X \ra=a_{ij} \text{ \rm for  }   \{i,j\} \in V\cup E, \ X \psd 0\},
\end{equation}
for any $a\in \mcs_+(G)$.
\end{itemize}
 \end{theorem}

\begin{proof}Taking orthogonal complements in \eqref{def:sap} we see that $M$ satisfies the SAP if and only if 
%${\mathcal T_M}^\perp \cap \lin\{E_{ij}: \{i,j\}\in V\cup E\}^\perp =\{0\}$ or, equivalently,
${\mathcal T_M}^\perp \cap\lin\{E_{ij}: \{i,j\}\in \overline E\}=\{0\}$.
Moreover, observe that the feasible region of the dual of the semidefinite program (\ref{eq:duala}) is equal to $\mcs^n_+\cap \mathcal C(G)$.
Now, using 
(\ref{eq:nondeg}), we obtain the equivalence of $(i), (ii)$ and $ (iii)$.
%have  that  $M$ is a nondegenerate solution for (\ref{eq:sap}) if and only if 
%$\mct_M+\lin \{E_{ij}: \{i,j\} \in V\cup E  \}=\mcs^n$ which by  \eqref{def:sap} is equivalent to  the SAP.\qed 
\end{proof}

\subsection{Graph parameters}\label{sec:param}

In this section we  explore the  relation between the two graph parameters $\gd(\cdot)$ and $\nu^=(\cdot)$ using the machinery developed in the previous sections. Recall that   $\mcs_+(G)$ denotes the set of $G$-partial psd matrices.
%We start with some basic definitions. We denote by  $\mce_n$ denote the set of positive semidefinite matrices whose diagonal entries are all equal to one. Given a graph $G=(V,E)$ let $\mce(G)=\pi_{E\cup V}(\mce_n)$.

\begin{definition}
 Given a graph $G=(V,E)$,  a vector $a\in\mcs_+(G) $ and an integer $k\ge 1$, a  {\em Gram representation} of $a$ in $\oR^k$ 
consists of  a set of vectors $p_1,\ldots,p_n\in \oR^k$  such that $$p_i^\sfT p_j=a_{ij}\  \text{ \rm for all }  \{i,j\} \in   V\cup E.$$
The Gram dimension of    $a\in \mcs_+(G)$, denoted as $\gd(G,a)$,  is  the smallest integer $k\ge 1$ for which $a$ has a Gram representation in $\oR^{k}$. %It will be denoted by  $\gd(G,a)$.
%The definition extends  to vectors $a\in \EE(G)$ in the obvious way: then all diagonal entries are implicitly taken to be equal to 1
%and  $\gd(G,a)$ is the smallest integer $k$ for which $a$ has a Gram representation in the unit sphere 
%$\bS^{k-1}$. %=\{x\in\oR^k\mid \|x\|=1\}$.
\end{definition}

\begin{definition} 
The {\em Gram dimension} of a graph $G$  is defined as
\begin{equation}\label{gramdimdef}
 \gd(G)=\underset{a \in \mcs_+(G)}{\max} \gd(G,a).
 \end{equation}
\end{definition}
This graph parameter was introduced and studied in~\cite{LV,LV12}, motivated by  its relevance to the low rank positive semidefinite matrix completion problem. 
Indeed, if   $G$ is a graph satisfying  $\gd(G)\le k$,   then  every $G$-partial psd matrix also has a psd completion of rank at most $k$. In~\cite{LV,LV12} the graph parameter $\gd(\cdot)$ is shown to be minor monotone and the graphs with small Gram dimension are characterized:
 $\gd(G)\le 2 \Longleftrightarrow  G$ is a forest (no $K_3$ minor),
 $\gd(G)\le 3 \Longleftrightarrow G$ is series-parallel (no $K_4$ minor),
 $\gd(G)\le 4 \Longleftrightarrow G$ has no $K_5$ and $K_{2,2,2}$ minors.

 %between  the Gram dimension of a graph  and  the graph parameter $\nu^=(G)$ introduced  in~\cite{H96,H03}.  
 
\medskip
 Next we  recall the definition of  the graph parameter $\nu^=(\cdot )$. 

\begin{definition}\cite{H96,H03}\label{nu} Given a graph $G=([n],E)$ the  parameter $\nu^=(G)$ is defined as the maximum corank of a matrix $M\in {\mathcal C}(G)\cap \mcs^n_+$ satisfying the SAP.
%$$\forall X\in \ms^n \ \ \ MX=0, \  X_{ii}=0 \   \forall i\in V, \ X_{ij}=0\  \forall (i,j)\in E \ \Longrightarrow X=0,$$
%known as the {\em strong Arnold property}.
%any positive semidefinite matrix $M \in \SSS^{n}_+$ such that  $M_{i,j}=0$ if $i\ne j\in V$ and $(i,j) \not \in E$  and $M$ satisfies the Strong Arnold Property.
\end{definition}

The study of the parameter $\nu^=(\cdot)$  is motivated by its relation to the  Colin de Verdi\`ere graph parameter $\mu(\cdot)$ mentioned above; for instance, $\mu(G)\le \nu^=(G)$ for any graph. In~\cite{H96,H03} it is shown that $\nu^=(\cdot)$ is minor monotone and the graphs with small value of $\nu^=(\cdot)$ are characterized:   $\nu^=(G)\le 2 \Longleftrightarrow  G$ is a forest (no $K_3$ minor),
 $\nu^=(G)\le 3 \Longleftrightarrow G$ is series-parallel (no $K_4$ minor),
 $\nu^=\le 4 \Longleftrightarrow G$ has no $K_5$ and $K_{2,2,2}$ minors.

\medskip
In view of the above  two characterizations it is  natural to try to  identify the exact relation between these  two graph parameters.  The following theorem is a first result in this direction.

%It is proven in~\cite{H96,H03} that $\nu^=(G)$ is a minor monotone graph parameter. Hence  for any fixed integer $k \ge 1$ the class of graphs with $\nu^=(G)\le k$  
%can be characterized by a finite family of minimal forbidden minors. For $k\le 3$ the only forbidden minor is $K_{k+1}$. Moreover, $\nu^=(G)\le 4$ if and only if $G$ does not have $K_5$ and $K_{2,2,2}$ as minors\cite{H96,H03}. In \cite{LV} the following relation was established between these two graph parameters.
 
\begin{theorem}\cite{LV}\label{thm:v_vs_gd} For any graph $G$, $\gd(G) \ge \nu^=(G)$.
\end{theorem}

It is not known  whether the two graph parameters coincide or not. %Theorem~\ref{thm:v_vs_gd} holds with equality. 
 We now  derive a new characterization of the parameter $\nu^=(\cdot)$ in terms of the maximum Gram dimension of certain $G$-partial psd matrices satisfying some nondegeneracy property,  which could be helpful to clarify the links between the two parameters. Recall that with  a vector $a \in \mcs_+(G)$ we can associate the following pair  of primal and dual  semidefinite programs:
  \begin{equation}\tag{$P_a$}\label{eq:sdp3}
\underset{X}{{\rm sup}} \left\{ 0 : \la E_{ij},X\ra=a_{ij} \ \text{ for } \{i,j\} \in V\cup E,   \text{ and } X \psd 0\right\},
\end{equation}
\begin{equation}\tag{$D_a$}\label{eq:sdp3d}
\underset{y,Z}{\inf}\{ \sum_{\{i,j\} \in V\cup E}y_{ij}a_{ij}: \sum_{\{i,j\} \in V\cup E}y_{ij}E_{ij}=Z\psd 0\}.
\end{equation}

Notice that, for any $a\in \mcs_+(G)$,  the primal program \eqref{eq:sdp3} is feasible and the dual program \eqref{eq:sdp3d} is strictly feasible. Thus there is no duality gap. 

\begin{definition}
Given a graph $G$, let   
$\mathcal{D}(G)$ denote the set of partial matrices $a\in \mcs_+(G)$ for which the semidefinite program \eqref{eq:sdp3d} has a nondegenerate optimal solution.
\end{definition}

We can now reformulate the parameter $\nu^=(G)$ as the maximum Gram dimension of a partial matrix in $\mathcal D(G)$.

\begin{theorem} \label{thmnewnu}
For any graph $G$ we have that 
$$ \nu^=(G)=\underset{a \in \mathcal{D}(G)}\max \gd(G,a).$$ 
\end{theorem}

\begin{proof}Suppose that $\underset{a \in \mathcal{D}(G)}\max \gd(G,a)=\gd(G,a^*)$.  As $a^* \in \mathcal{D}(G)$  it follows that $(D_{a^*})$ has  a nondegenerate optimal solution which we denote by $M$.  Then, Theorem~\ref{thm:uprimal} implies that~$(P_{a^*})$ has a unique solution which we denote by  $A$. Notice that the matrix $A$ is the unique psd  completion of the partial matrix $a^*\in \mcs_+(G)$ which implies that $\gd(G,a^*)=\rankspace A$. Moreover, as $A$ and $Z$ are a pair of primal dual optimal solutions we have that $AM=0$ which implies that $\cor M \ge~\rankspace A$. As the matrix $M$ is feasible for $\nu^=(G)$ (recall Definition~\ref{nu} and Theorem~\ref{thm:SAPtosdp}) it follows that $\nu^=(G)\ge \underset{a \in \mathcal{D}(G)}\max \gd(G,a)$. 

For the other direction, assume $\nu^=(G)=\cor M=d$
where $ M\in~\mcc(G)\cap \mcs^n_+$ and $M$  satisfies the SAP. Let    $P\in \mr^{n\times d}$ be a matrix whose   columns form a basis for $\Ker M$ and consider the partial matrix $a \in \mcs_+(G)$ defined as $a_{ij}=(PP^\sfT)_{ij}$ for every $\{i,j\} \in V\cup E$. As $\la M,PP^\sfT \ra=0$ it follows  that $M$ is a dual nondegenerate optimal solution for \eqref{eq:sdp3d} and thus $a \in \mathcal{D}(G)$.  Additionally,   as $\cor M =\rankspace PP^\sfT$ we have that   $M$ and $PP^\sfT$ are a pair of strict complementary optimal solutions for~\eqref{eq:sdp3} and~\eqref{eq:sdp3d}, respectively.  Then Theorem~\ref{thm:converse} implies that the matrix $PP^\sfT$ is the   unique optimal solution of \eqref{eq:sdp3} and thus $\gd(G,a)=\rankspace PP^{\sfT}=\cor M=\nu^=(G)$.\qed
\end{proof}

\begin{corollary}\label{thm:gdnu}
 For any graph $G$,  we have that $\gd(G)\ge \nu^=(G)$. Moreover, equality $\gd(G)=\nu^=(G)  $ holds if and only if there exists some $a  \in \mathcal{D}(G)$ for which $\gd(G)=\gd(G,a)$.
\end{corollary}

%As observed above the two parameters coincide: $\gd(G)=\nu^=(G)$ for graphs with no $K_5$ and $K_{2,2,2}$ minors. It is also easy to see that they coincide, e.g., for chordal graphs.

%\subsection*{Acknowledgements.}

\end{document}